\documentclass[12pt,a4paper]{amsart}
\usepackage[top=30truemm,bottom=30truemm,left=25truemm,right=25truemm]{geometry}
\usepackage{amsfonts,amssymb,amscd,amsmath,bbm,latexsym,amsbsy, mathdots, fancyhdr}
\usepackage[utf8]{inputenc}
\usepackage{calc}
\usepackage{accents}
\usepackage[demo]{graphicx}
\usepackage{caption}
\usepackage{subcaption}
\usepackage{cleveref}

\usepackage{tikz}
	\usetikzlibrary{graphs,arrows,decorations.pathmorphing,backgrounds,positioning,fit,chains,matrix,scopes,decorations.markings}

\tikzset{-<-/.style={decoration={ 
  markings,
  mark=at position .5 with {\arrowreversed[line width = 0.5pt]{angle 90}}},postaction={decorate}}}
        
\tikzset{->-/.style={decoration={ 
  markings,
  mark=at position .5 with {\arrow[line width = 0.5pt]{angle 90}}},postaction={decorate}}}

\usepackage{tikz}

\numberwithin{equation}{section}
\theoremstyle{plain}
\newtheorem{thm}{Theorem}[section]

\newtheorem{prop}[thm]{Proposition}

\newtheorem{lem}[thm]{Lemma}
\newtheorem{cor}[thm]{Corollary}

\theoremstyle{remark}
\newtheorem{rema}[thm]{Remark}

\newcommand{\dbtilde}[1]{\accentset{\approx}{#1}}

\newcommand{\mydots}{\hbox to 1em{.\hss.\hss.}}

\newcommand{\bA}{\mathbf{A}}
\newcommand{\bc}{\mathbf{c}}
\newcommand{\bd}{\mathbf{d}}

\newcommand{\bz}{\mathbf{z}}

\newcommand{\bnull}{{\mathbf{0}}}
\newcommand{\C}{{\mathbb C}}

\newcommand{\N}{{\mathbb N}}
\newcommand{\cA}{{\mathcal A}}

\newcommand{\cB}{{\mathcal B}}

\newcommand{\cF}{{\mathcal F}}

\newcommand{\cG}{{\mathcal G}}

\newcommand{\cI}{{\mathcal I}}

\newcommand{\cJ}{{\mathcal J}}

\newcommand{\cM}{{\mathcal M}}

\newcommand{\cX}{{\mathcal X}}
\newcommand{\cY}{{\mathcal Y}}
\newcommand{\cW}{{\mathcal W}}
\newcommand{\cZ}{{\mathcal Z}}

\newcommand{\eps}{\varepsilon}

\newcommand{\field}{{\mathbb K}}

\newcommand{\gfrak}{{\mathfrak g}}

\newcommand{\glfrak}{{\mathfrak{gl}}}

\newcommand{\gr}{{\mathrm{gr}}}
\newcommand{\hfrak}{{\mathfrak h}}

\newcommand{\Hom}{{\mathrm{Hom}}}

\newcommand{\id}{{\mbox{id}}}
\newcommand{\im}{{\mbox{Im}}}

\newcommand{\kfrak}{{\mathfrak k}}
\newcommand{\kow}{{\varDelta}}

\newcommand{\nfrak}{{\mathfrak n}}

\newcommand{\ob}{{\overline{b}}}

\newcommand{\ov}{{\overline{v}}}

\newcommand{\one}{\mathbf{1}}

\newcommand{\ot}{\otimes}

\newcommand{\qfield}{k}

\newcommand{\Q}{\mathbb Q}

\newcommand{\rank}{\mathrm{rank}}

\newcommand{\rt}{\mathrm{rt}}

\newcommand{\slfrak}{{\mathfrak{sl}}}
\newcommand{\sofrak}{{\mathfrak{so}}}

\newcommand{\ttJ}{\dbtilde{J}}
\newcommand{\ttL}{\dbtilde{L}}
\newcommand{\ttN}{\dbtilde{N}}

\newcommand{\Uq}{U}

\newcommand{\uqg}{{U_q(\mathfrak{g})}}

\newcommand{\vep}{\varepsilon}

\newcommand{\Z}{{\mathbb Z}}

\begin{document}
\title{Representation theory of very non-standard quantum ${\sofrak}(2N-1)$}

\author{Stefan Kolb}
\address{Stefan Kolb and Jake Stephens, School of Mathematics, Statistics and Physics, Newcastle University, Herschel Building, Newcastle upon Tyne NE1 7RU, UK}
\email{stefan.kolb@newcastle.ac.uk}
\email{jakehenrystephens@gmail.com}

\author{Jake Stephens}



\begin{abstract}
  We classify the finite dimensional representations of the quantum symmetric pair coideal subalgebra $\cB_\bc$ of type $DII$ corresponding to the symmetric pair $(\sofrak(2N),\sofrak(2N-1))$. For $\cB_\bc$ defined over an arbitrary field $\qfield$ and $q\in \qfield$ not a root of unity we establish a one-to-one correspondence between finite dimensional, simple $\cB_\bc$-modules and dominant integral weights for $\sofrak(2N-1)$. We use specialisation to show that the category of finite dimensional $\cB_\bc$-modules is semisimple if $\mathrm{char}(\qfield)=0$ and $q$ is transcendental over $\Q$. In this case the characters of simple $\cB_\bc$-modules are given by Weyl's character formula. This means in particular that the quantum symmetric pair of type $DII$ can be used to obtain Gelfand-Tsetlin bases for irreducible representations of the Drinfeld-Jimbo quantum group $U_q(\sofrak(2N))$.
\end{abstract}  
\maketitle
\section{Introduction}
\subsection{Motivation and background} Let $\gfrak_n$ denote the general linear Lie algebra $\glfrak(n)$ of complex $n\times n$ matrices or the complex orthogonal Lie algebra $\sofrak(n)$. There are natural embeddings of Lie algebras $\gfrak_n\hookrightarrow \gfrak_{n+1}$.
In their seminal work \cite{a-GT50a}, \cite{a-GT50b} I.~Gelfand and M.~Tsetlin developed explicit formulas for representations of $\gfrak_n$ in terms of combinatorial data. These Gelfand-Tsetlin patters can be interpreted as sequences $\lambda_n, \lambda_{n-1}, \dots, \lambda_2, (\lambda_1)$,
where $\lambda_\ell$ is a dominant integral weight for $\gfrak_\ell$, such that the simple $\gfrak_\ell$-module of highest weight $\lambda_\ell$ is contained inside the simple $\gfrak_{\ell+1}$-module of highest weight $\lambda_{\ell+1}$ under restriction. As restriction from $\gfrak_{\ell+1}$ to $\gfrak_\ell$ is multiplicity free, the sequence $\lambda_n, \lambda_{n-1}, \dots, \lambda_2, (\lambda_1)$ determines a basis of the simple $\gfrak_n$-module of highest weight $\lambda_n$ up to scalar factors.

With the advent of quantum groups in the 1980s, it was natural to ask if similar bases exist for representations of Drinfeld-Jimbo quantised enveloping algebras $U_q(\gfrak_n)$. For $\gfrak_n=\glfrak(n)$ this was settled early on \cite{a-Jimbo3}, \cite{a-UST90}. However, for orthogonal Lie algebras there is no embedding of $U_q(\sofrak(n-1))$ into $U_q(\sofrak(n))$ and therefore Gelfand-Tsetlin bases for $U_q(\sofrak(n))$ were long deemed non-existent. This observation was the starting point for A.~Gavrilik and A.~Klimyk's invention of non-standard $U'_q(\sofrak(n))$ in \cite{a-GavKlim91}. These algebras support inclusions $U'_q(\sofrak(n-1))\subset U'_q(\sofrak(n))$ and Gavrilik and Klimyk constructed finite dimensional, irreducible representations of $U'_q(\sofrak(n))$ in terms of Gelfand-Tsetlin patterns for $\sofrak(n)$. 

M.~Noumi observed in \cite{a-Noumi96} that Gavrilik and Klimyk's $U'_q(\sofrak(n))$ can be realised as a (right) coideal subalgebra of Drinfeld-Jimbo $U_q(\slfrak(n))$, and he established the pair of algebras $(U_q(\slfrak(n)), U'_q(\sofrak(n)))$ as one of the first examples of a quantum symmetric pair. 
Noumi and collaborators constructed quantum symmetric pairs for all classical symmetric Lie algebras, that is, for all involutive Lie algebra automorphisms $\theta:\gfrak\rightarrow \gfrak$ with $\gfrak$ of classical type \cite{a-Noumi96}, \cite{a-NS95}, \cite{a-NDS97}. Their construction proceeds case by case and relies on explicit solutions of the reflection equation.

Independently, G.~Letzter devised a uniform construction of quantum symmetric pairs for all semisimple symmetric Lie algebras \cite{a-Letzter99a}. Recall that involutive automorphisms $\theta$ of a semisimple, complex Lie algebra $\gfrak$ are classified up to conjugation in terms of Satake diagrams $(X,\tau)$, see also Section \ref{sec:symSatake}. Here $X$ is a subset of the set of nodes of the Dynkin diagram of $\gfrak$ and $\tau$ is a diagram automorphism preserving the subset $X$. For each Satake diagram, Letzter defined a right coideal subalgebra $\cB\subset \uqg$ which is a quantum group analogue of $U(\kfrak)$ for $\kfrak=\{x\in \gfrak\,|\,\theta(x)=x\}$. We refer to $(\uqg,\cB)$ as a quantum symmetric pair (QSP) and to the algebra $\cB$ as a QSP-coideal subalgebra of $\uqg$. For $\gfrak$ of classical type, Letzter's construction reproduces the coideal subalgebras constructed by Noumi and collaborators, see \cite[Section 6]{a-Letzter99a}.

The pair of Lie algebras $(\sofrak(n),\sofrak(n-1))$ is a symmetric pair. If $n=2N$ is even, then the corresponding Satake diagram is of type $DII$, see \Cref{fig:DII}.
\begin{figure}
\centering
\begin{subfigure}[t]{.5\textwidth}
  \centering
    \begin{tikzpicture}
			[scale=0.7, white/.style={circle,draw=black,inner sep = 0mm, minimum size = 3mm},
			black/.style={circle,draw=black,fill=black, inner sep= 0mm, minimum size = 3mm}, every node/.style={transform shape}]	
			\node[white] (first) [label = below:{\scriptsize $1$}] {};		
			\node[black] (second) [right= of first] [label = below:{\scriptsize $2$}]  {}
				edge (first);
			\node[black] (third) [right = 1.5cm of second] {}
				edge [dashed] (second);
			\node[black] (fourth) [above right = 0.5cm of third] [label = above:{\scriptsize $N-1$}] {}
				edge (third);
			\node[black] (fifth) [below right = 0.5cm of third] [label = below:{\scriptsize $N$}] {}
				edge (third);		
		\end{tikzpicture}
                \caption{$n=2N$, type $DII$}\label{fig:DII}
\end{subfigure}%
\begin{subfigure}[t]{.5\textwidth}
  \centering
   \begin{tikzpicture} 
			[scale=0.7, white/.style={circle,draw=black,inner sep = 0mm, minimum size = 3mm},
			black/.style={circle,draw=black,fill=black, inner sep= 0mm, minimum size = 3mm}, every node/.style={transform shape}] 
			\node[white] (first)  [label = below:{\scriptsize $1$}] {};		
			\node[black] (second) [right=of first] [label = below:{\scriptsize $2$}]  {}			
				edge (first);
			\node[black] (third) [right = 1.5cm of second]  {}
				edge [dashed] (second);
			\node[black] (fourth) [right = of third] [label = below:{\scriptsize $N$}] {}
			edge [double equal sign distance, -<-] (third);
                        \phantom{	\node[black] (fifth) [below right = 0.5cm of third] [label = below:{\scriptsize $N$}] {}
				edge (third);	}
   \end{tikzpicture}
   
     \caption{$n=2N+1$, type $BII$}\label{fig:BII}
  \end{subfigure}
\caption{Satake diagrams for $(\sofrak(n),\sofrak(n-1))$}
\label{fig:Satake}
\end{figure}
Quantum symmetric pairs provide a corresponding coideal subalgebra $\cB_{2N-1}\subset U_q(\sofrak(2N))$ which is a quantum group analogue of $U(\sofrak(2N-1))$. The algebra $\cB_{2N-1}$ is different from Gavrilik and Klimyk's non-standard $U_q'(\sofrak_{2N-1})$ and hence we also refer to it as `very non-standard' quantum $\sofrak(2N-1)$. By Letzter's construction we obtain a chain of algebras
\begin{align}\label{eq:chain1}
   \dots \subset U_q(\sofrak(2N-2)) \subset \cB_{2N-1} \subset U_q(\sofrak(2N))\subset \dots 
\end{align}
In the odd case $n=2N+1$, the Satake diagram corresponding to the symmetric pair $(\sofrak(2N+1),\sofrak(2N))$ is of type $BII$, see \Cref{fig:BII}. Again, quantum symmetric pairs provide a coideal subalgebra $\cB_{2N}\subset U_q(\sofrak(2N+1))$ which is a quantum group analogue of $U(\sofrak(2N))$, and we obtain a chain of algebras
\begin{align}\label{eq:chain2}
   \dots \subset U_q(\sofrak(2N-1)) \subset \cB_{2N} \subset U_q(\sofrak(2N+1))\subset \dots 
\end{align}
It is natural to expect that the chains of algebras \eqref{eq:chain1} and \eqref{eq:chain2} lead to Gelfand-Tsetlin bases for representations of $U_q(\sofrak(n))$. To this end, we need to understand the representation theory of the QSP-coideal subalgebras $\cB_{2N-1}$ and $\cB_{2N}$. The two situations are quite different as they correspond to different Satake diagrams. In the present paper we consider the quantum symmetric pair $(U_q(\sofrak(2N)),\cB_{2N-1})$.

\subsection{Representations of QSP-coideal subalgebras}
 If $(\gfrak, \kfrak)$ is a symmetric pair with semisimple $\gfrak$ then $\kfrak$ is a reductive Lie algebra. 
 Hence it is natural to expect that any QSP-coideal subalgebra $\cB$ has a nice category of finite-dimensional representations. However, to this date, finite dimensional representations of $\cB$ have been classified only for a fairly limited class of examples. To put the results of the present paper into context, we provide a brief overview of known results.

 In type $AI$, corresponding to the symmetric pair $(\slfrak(n),\sofrak(n))$, finite dimensional representations of $\cB$ were classified inductively by N.~Iorgov and A.~Klimyk \cite{a-IK05}. There appear two types of irreducible representations, those of classical type which allow specialisation $q\rightarrow 1$ and those of nonclassical type which don't. The classification relies on the explicit realisation of these representations given in \cite{a-GavKlim91}. These realisations involve square roots of $q$-numbers and hence \cite{a-IK05} work over the field $\qfield=\C$ with $q\in \C$ not a root of unity. The induction argument in \cite{a-IK05} also shows that the category of finite dimensional $\cB$-modules is semisimple in this case.

 The classification in type $AI$ was revisited by H.~Wenzl who developed a Verma module approach in this case \cite{a-Wenzl20}. In particular, Wenzl identifies large proper submodules of Verma modules and can show by an $\slfrak(2)$-argument similar to the classical case that the corresponding quotients are finite dimensional if the highest weight is dominant integral \cite[Section 5.3]{a-Wenzl20}. This is possible because the rank two case provides $\slfrak(2)$-triples in type $AI$.

In type $AII$, corresponding to the symmetric pair $(\slfrak(2N)),\mathfrak{sp}(2N))$, irreducible representations of $\cB$ were classified by A.~Molev in \cite{a-Molev06}. Molev works essentially with Noumi's construction of $\cB$ and again follows a Verma module approach. To show that simple quotients of Verma modules for dominant integral weights are finite dimensional, Molev uses Gelfand-Tsetlin bases for $U_q(\glfrak(2N))$ to construct non-vanishing homomorphisms from Verma modules for $\cB$ to finite dimensional $U_q(\glfrak(2N))$-modules. Moreover, Molev uses specialisation to show that in type $AII$ the characters of simple $\cB$-modules are given by Weyl's character formula. 

A unifying approach to the classification of finite dimensional $\cB$-modules was proposed by H.~Watanabe in \cite{a-Watanabe21}. To include all types, he works over the algebraic closure $\overline{\C(q)}$. Watanabe defines a class of modules of classical weight which excludes the nonclassical modules of type $AI$ and which is not always closed under tensor products with 1-dimensional $\uqg$-modules. For this class Watanabe develops a highest weight theory. To make this theory work, one needs to establish a triangular decomposition of $\cB$ in terms of analogues of root vectors in $\cB$, which is quite challenging in general. In \cite{a-Watanabe21} this program is performed for several classes of examples, reproving the results of \cite{a-Molev06}, parts of the classification in \cite{a-IK05} and \cite{a-Wenzl20}, and leading to a classification of simple modules of classical weight in type AIII corresponding to the symmetric pairs $(\slfrak(2n),\mathfrak{s}(\glfrak(n)\oplus \glfrak(n)))$ and $(\slfrak(2n+1),\mathfrak{s}(\glfrak(n)\oplus \glfrak(n+1)))$.
To show that simple quotients of Verma modules for dominant integral weights are finite dimensional, Watanabe uses specialisation early on. Specialisation is also employed to show that any finite dimensional $\cB$-module is completely reducible in those cases.
\subsection{Results}
In the present paper, we classify all finite dimensional simple $\cB$-modules in type $DII$ corresponding to the symmetric pair $(\sofrak(2N),\sofrak(2N-1))$. 
We aim to emulate the approach in Jantzen's textbook \cite[Chapter 5]{b-Jantzen96}. In particular, we work over any field $\qfield$ with $q\in \qfield$ not a root of unity. We follow the conventions in \cite{a-Kolb14} and write $\cB=\cB_\bc$ to indicate the parameter dependence of the QSP-coideal subalgebra of type $DII$.

Let $\Uq^0\subset U_q(\sofrak(2N))$ be the subalgebra generated by the group-like elements $K_i^{\pm 1}$ for $i=1,\dots,N$. In type $DII$, the number of black dots of the Satake diagram $|X|=N-1$ coincides with the rank of $\kfrak=\sofrak(2N-1)$. Hence, the QSP-coideal subalgebra $\cB_\bc$ has a natural analogue of a Cartan subalgebra, namely $\Uq^0_X:=\cB_\bc \cap \Uq^0=\qfield\langle K_i^{\pm 1}\,|\,i\in X\rangle$. This only happens for symmetric pairs of type $AII$, $DII$ and $EIV$ and simplifies the situation significantly.

The algebra $\cB_\bc$ is generated by a single element $B_1$ and the subalgebra $\cM_X=U_q(\gfrak_X)$ corresponding to the set $X$ of black dots in the Satake diagram. Taking inspiration from the classical case, we define root vectors
for $\cB_\bc$ in terms of the generator $B_1$ and the Lusztig braid automorphism $T_i$ for $i\in X$. The root vectors are ordered by a choice of longest element in the Weyl group $W$. Moreover, they split up into positive and negative root vectors depending on their behaviour under conjugation by $U^0_X$. Let $\cB_\bc^+$ and $\cB_\bc^-$ be the span of all ordered monomials in the positive and negative root vectors of $\cB_\bc$, respectively. We obtain a triangular decomposition for $\cB_\bc$.

\medskip

\noindent{\bf Theorem A.} (Theorem \ref{thm:PBW-Bc}, Corollary \ref{cor:Bc-triang})
\textit{The multiplication map
\begin{align*}
  \cB_\bc^- \ot U^0_X \ot \cB_\bc^+ \rightarrow \cB
\end{align*}
is a linear isomorphism. Moreover, the ordered monomials in the positive (negative) root vectors form a basis of $\cB_\bc^+$ ($\cB_\bc^-$).}

\medskip

Any finite dimensional $\cB_\bc$-module $V$ can be considered as a $\cM_X$-module by restriction, and hence the action of $U^0_X$ on $V$ is diagonalisable. The notion of type of a finite dimensional $\cM_X$-module \cite[5.2]{b-Jantzen96} lifts to a notion of type of a finite dimensional $\cB_\bc$-module. As in the quantum group case, it suffices to consider finite dimensional $\cB_\bc$-modules of type $\mathbf 1$, as all others can be obtained by tensoring by a one-dimensional $U_q(\sofrak(2N))$-module, see Proposition \ref{prop:cat-equiv}.

Let $P_{2N-1}$ denote the weight lattice of $\sofrak(2N-1)$ and let $P_{2N-1}^+$ be the subset of dominant integral weights. Finite dimensional $\cB_\bc$-modules decompose into weight spaces and hence contain a highest weight vector of weight $\lambda\in P_{2N-1}$ on which $\cB_\bc^+$ acts trivially. We need to show that $\lambda$ is dominant. Let $\gamma_1,\dots,\gamma_{N-1}$ be the simple roots of $\sofrak(2N-1)$ with $\gamma_{N-1}$ being the short root. For $i=1,\dots,N-2$ the triple $\{E_{i+1}, F_{i+1}, K_{i+1}^{\pm 1}\} \subset \cB_\bc$ generates a copy of $U_q(\slfrak(2))$ corresponding to the root $\gamma_i$, and hence $\lambda(h_{\gamma_i})\ge 0$ where $h_{\gamma_i}$ is the corresponding coroot. We denote the positive and negative root vectors corresponding to the short root $\gamma_{N-1}$ by $B_{\beta_{N-1}}$ and $B_{\beta_N}$, respectively. We show in Proposition \ref{prop:Uqsl2} that the triple $\{B_{\beta_{N-1}}, B_{\beta_N}, (K_N K_{N-1}^{-1})^{\pm 1}\}$ acts on a suitable subspace $H(V)\subset V$ as a copy of $U_q(\slfrak(2))$. Moreover, $H(V)$ contains the highest weight vector. This implies that $\lambda(h_{\gamma_{N-1}})\ge 0$ and hence $\lambda\in P^+_{2N-1}$ is dominant.

Using the root vectors for $\cB_\bc$ we define a Verma module $M(\lambda)$ for each $\lambda\in P_{2N-1}$, see Section \ref{sec:Verma}. The Verma module $M(\lambda)$ is generated by a highest weight vector $v_\lambda$ and we show in Theorem \ref{thm:PBWMl} that the linear map
\begin{align*}
  \cB_\bc^-\rightarrow M(\lambda), \qquad b\mapsto bv_\lambda
\end{align*}
is a linear isomorphism. 
Any finite dimensional simple $\cB_\bc$-module is a quotient of $M(\lambda)$ by a proper maximal submodule $N(\lambda)\subset M(\lambda)$ for some $\lambda\in P^+_{2N-1}$. To complete the classification of finite dimensional, simple $\cB_\bc$-modules, it remains to show that
\begin{align*}
  L(\lambda) = M(\lambda)/N(\lambda) 
\end{align*}
is finite dimensional for all $\lambda\in P_{2N-1}^+$. We first consider the submodule
\begin{align*}
  \widetilde{N}(\lambda) = \sum_{i=1}^{N-2} \cB_\bc F_{i+1}^{n_i+1}v_\lambda +\cB_\bc B_{\beta_{N}}^{n_{N-1}+1}v_\lambda
\end{align*}
with $n_i=\lambda(h_{\gamma_i})\ge 0$. As in the quantum group case, we show that 
$\widetilde{N}(\lambda) \subset M(\lambda)$ is a proper submodule. However, it is unclear how to adapt the Weyl group invariance argument from \cite[5.9]{b-Jantzen96}, since $\{B_{\beta_{N-1}}, B_{\beta_N}, (K_N K_{N-1}^{-1})^{\pm 1}\}$ acts as an $\slfrak(2)$-triple only on the subspace $H(V)$. Instead we consider the submodule
\begin{align*}
  \ttN(\lambda) = \sum_{i=1}^{N-2} \cB_\bc F_{i+1}^{n_i+1}v_\lambda +\cB_\bc F_N^{n_N+1}v_\lambda
\end{align*}
with $n_\lambda=\lambda(h_{\gamma_{N-2}} + h_{\gamma_{N-1}})$ which contains $\widetilde{N}(\lambda)$, see Remark \ref{rem:NtNtt}. By Proposition \ref{prop:FNvl} we know that $\ttN(\lambda)$ is also a proper submodule of $M(\lambda)$. Using a filtered-graded argument, we show in Section \ref{sec:filt-grad} that $M(\lambda)/\ttN(\lambda)$ is finite dimensional. This gives us the desired classification.

\medskip

\noindent{\bf Theorem B.} (Theorem \ref{thm:Llambda}) \textit{Let $\qfield$ be any field and $q\in \qfield$ not a root of unity.
  For each $\lambda\in P_{2N-1}^+$ the simple $\cB_\bc$-module $L(\lambda)$ has finite dimension. Each finite-dimensional simple $\cB_\bc$-module (of type ${\bf 1}$) is isomorphic to exactly one $L(\lambda)$ with $\lambda\in P_{2N-1}^+$.}

\medskip

We would like to know the characters of the simple $\cB_\bc$-modules $L(\lambda)$ and we would like to show that all finite dimensional $\cB_\bc$-modules are semisimple. At this point, we can only achieve this via specialisation, as in \cite[5.15, 5.17]{b-Jantzen96}, \cite{a-Molev06}, \cite{a-Watanabe21}. We obtain the following result.

\medskip

\noindent{\bf Theorem C.} (Corollary \ref{cor:Weyl}, Theorem \ref{thm:semisimple}) Assume that $\qfield$ is a field of characteristic zero and that $q$ is transcendental over $\Q$. Then every finite dimensional $\cB_\bc$-module is semisimple. Moreover, for any $\lambda\in P^+_{2N-1}$ the submodules $N(\lambda), \widetilde{N}(\lambda)$ and $\ttN(\lambda)$ coincide, and the dimensions of the weight spaces $L(\lambda)_\mu$ for $\mu\in P_{2N-1}$ are given by Weyl's character formula.

\medskip

Theorem C implies in particular that finite dimensional $U_q(\sofrak(2N))$-modules decompose into irreducible $\cB_\bc$-modules as in the classical case $q=1$. And similarly, finite-dimensional $\cB_\bc$-modules decompose into irreducible $U_q(\sofrak(2N-2))$-modules as in the classical case $q=1$. Hence, if $\mathrm{char}(\qfield)=0$ and $q$ is transcendental over $\Q,$ then the chain of inclusions \eqref{eq:chain1} provides a Gelfand-Tsetlin basis of any simple $U_q(\sofrak(2N))$-module, up to scalar factors.
\subsection{Very recent developments}
The theory of quantum symmetric pairs is presently developing at rapid speed. While we were putting final touches to the present paper, two preprints appeared which have bearing on our results. In \cite{a-Watanabe24p} H.~Watanabe proposed a notion of integrable modules for QSP-coideal subalgebras. He shows  that simple, integrable modules over QSP-coideal subalgebras of finite type are finite dimensional \cite[Section 5.1]{a-Watanabe24p}. In the spirit of our paper to emulate the constructions of \cite[Chapter 5]{b-Jantzen96} as closely as possible, it would be desirable to show that the quotient $M(\lambda)/\ttN(\lambda)$ for $\lambda\in P^+_{2N-1}$ is integrable in the sense of Watanabe's paper.

M.~Lu, R.~Yang and W.~Zhang proposed notions of root vectors and PBW bases for QSP-coideal subalgebras of finite type \cite{a-LuYangZhang24p}. In type $DII$ the root vectors given in \cite[3.7]{a-LuYangZhang24p} coincide with the root vectors given in Section \ref{sec:RV4Bc}. Moreover, the second part of Theorem A, Theorem \ref{thm:PBW-Bc}, is a special case of \cite[Theorem 4.4]{a-LuYangZhang24p}. We note that the definition of the root vectors in type $DII$ and the PBW-Theorem \ref{thm:PBW-Bc} already appeared in the second named author's PhD thesis \cite{phd-Stephens23}.

\medskip

\noindent {\bf Acknowledgements.} The first named author is grateful to Johnathan Brundan and Ben Webster for helpul comments at an early stage of this project.

\section{The symmetric pair $ (\sofrak(2N), \sofrak(2N-1))$}\label{sec:sym-pair}
In this preliminary section we identify the orthogonal Lie algebra $\sofrak(2N-1)$ explicitly as the Lie subalgebra  pointwise fixed under a maximally split involution on $\sofrak(2N)$. This allows us to identify simple root vectors of $\sofrak(2N-1)$ inside $\sofrak(2N)$, which will in turn inform the choice of root vectors in the quantum setting in Section \ref{sec:PBW}. 
\subsection{Symmetric pairs and Satake diagrams}\label{sec:symSatake}
Let $\gfrak$ be a semisimple complex Lie algebra with Cartan subalgebra $\hfrak$, root system $\Phi$, a fixed choice of simple roots $\Pi=\{\alpha_i\,|\,i\in I\}$, and corresponding triangular decomposition $\gfrak=\nfrak^+\oplus \hfrak \oplus \nfrak^-$. Let $e_i, f_i, h_i$ for $i\in I$ denote the Chevalley generators of $\gfrak$. For any subset $X\subset I$ let $\gfrak_X$ denote the Lie subalgebra of $\gfrak$ generated by $e_i, f_i, h_i$ for all $i\in X$. Let $Q_X=\sum_{i\in X}\Z \alpha_i$ be the corresponding root lattice and $Q_X^+=\sum_{i\in X}\N_0\alpha_i$ the positive cone. For $X=I$ we write $Q$ and $Q^+$ instead of $Q_I$ and $Q_I^+$, respectively.

Let $\theta:\gfrak\rightarrow \gfrak$ be an involutive Lie algebra automorphism, i.e.~$\theta^2=\id_\gfrak$, and $\kfrak=\{x\in \gfrak\,|\,\theta(x)=x\}$ the pointwise fixed Lie subalgebra of $\gfrak$. We refer to $(\gfrak,\kfrak)$ as a symmetric pair. Up to conjugation by a Lie algebra automorphism of $\gfrak$, the involution $\theta$ has the following properties, see \cite[1.13]{b-Dixmier96}, also \cite[Section 7]{MSRI-Letzter}:
\begin{enumerate}
  \item $\theta(\hfrak)=\hfrak$,
  \item There exists $X\subset I$ such that $\theta|_{\gfrak_X}=\id_{\gfrak_X}$,
  \item $\theta(e_i)\in \nfrak^-$ and $\theta(f_i)\in \nfrak^+$ if $i\in I\setminus X$, where $X$ is the set given in (2).  
\end{enumerate}
Following \cite{MSRI-Letzter}, we call the involution $\theta$ maximally split if properties (1), (2) and (3) hold. 
From now on we assume that $\theta$ is maximally split. By condition (1), the map $\theta$ induces an involution $\Theta:\hfrak^\ast\rightarrow \hfrak^\ast$ which leaves the root system $\Phi$ invariant. By construction we have $\theta(\gfrak_\alpha)=\gfrak_{\Theta(\alpha)}$ for all roots $\alpha\in \Phi$. Moreover, if $\theta$ is maximally split, then there exists an involutive diagram automorphism $\tau:I\rightarrow I$ such that $\Theta(-\alpha_i)-\alpha_{\tau(i)}\in Q^+_X$ for all $i\in I\setminus X$, see also \cite[(7.5)]{MSRI-Letzter}. The pair $(X,\tau)$ is called a Satake diagram and determines the involution $\theta$ uniquely up to conjugation, see also \cite[Theorem 2.7]{a-Kolb14}. The Satake diagram $(X,\tau)$ is recorded in the Dynkin diagram of $\gfrak$ by colouring all nodes in $X$ black and indicating the diagram automorphism by a double-pointed arrow. See \cite[p.~32/33]{a-Araki62} for a list of all Satake diagrams for simple Lie algebras $\gfrak$.  
\subsection{Root system and Chevalley generators for $\sofrak(n)$}\label{sec:rootChevalley}
For any $n\in \N$ let $\glfrak(n)$ denote the Lie algebra of complex $n\times n$-matrices. For any $m\times n$ matrix  $M$ we write $M^t$ to denote its transpose matrix. The complex orthogonal Lie algebra $\sofrak(n)$ can be defined as
\begin{align}\label{eq:glSn}
  \glfrak_S(n)=\{M\in \glfrak(n)\,|\,M= - S M^t S^{-1}\}
\end{align}
for any non-degenerate symmetric $n\times n$-matrix $S$. The Lie algebras
$\glfrak_S(n)$ for different choices of $S$ are all isomorphic by the classification of complex symmetric bilinear forms.
For the purpose of this paper, it is convenient to choose
\begin{align}\label{eq:S-def}
    S= S_n =\begin{pmatrix}  &  & 1 \\  &  \iddots & \\ 1 & & \end{pmatrix}=\sum_{i=1}^n E_{i,n-i+1}
\end{align}
where $E_{i,j}$ denotes the elementary $n\times n$ matrix with entry $1$ in the $ij$-th position and zero entries elsewhere. For this choice of $S$, the set $\hfrak_n$ of diagonal matrices in $\glfrak_S(n)$ is a Cartan subalgebra. Moreover, the definition of $S=S_n$ is uniform for  even and odd $n$. Throughout this paper, the symbol $\sofrak(n)$ for $n\in \N$ will always denote the Lie algebra $\glfrak_S(n)$ for $S$ given by \eqref{eq:S-def}.

For any $m\times n$ matrix $A$ we write $A^\tau=S_n A^t S_m$ to denote the matrix obtained by reflection in the antidiagonal. With this notation we have
\begin{align*}
  \sofrak(n)=\glfrak_S(n)=\{M\in \glfrak(n)\,|\,M=-M^\tau\}.
\end{align*}  
\begin{rema}\label{rem:JantzenJacobson}
  Let $\field$ be any field of characteristic $0$. Then \eqref{eq:glSn} with $S$ given by \eqref{eq:S-def} defines a Lie algebra $\glfrak_n(S)_\field$ over $\field$. The set $\hfrak_n$ of diagonal matrices in $\glfrak_S(n)_\field$ is a splitting Cartan subalgebra. Up to a change of basis this example is discussed in \cite[IV, Theorems 7 and 9]{b-Jacobson62}. Moreover, finite dimensional representations of $\glfrak_S(n)_\field$ are completely reducible, irreducible representations are given by dominant integral weights, and the characters of irreducible $\glfrak_n(S)$-modules are given by Weyl's character formula, see \cite{b-Jacobson62}. We will need these statements in Section \ref{sec:mod-special} when we prove analogous results for very non-standard quantum $\sofrak(2N-1)$ over a field of characteristic $0$ and $q$ transcendental over $\Q$, see also \cite[Remark 5.14]{b-Jantzen96}.
\end{rema}
For later reference, we describe root systems and Chevalley generators of $\sofrak(n)$ for even and for odd $n$ separately. Moreover, it is convenient to give block decompositions of $\sofrak(n)$ for even and odd $n$ which will allow us to realise explicit embeddings of $\sofrak(n-1)$ into $\sofrak(n)$.

\noindent{\bf Case $n=2N$ even:} In this case, we can write
\begin{align}\label{eq:so2N}
  \sofrak(2N)=\glfrak_S(2N)=\left\{\begin{pmatrix} A& B\\ C & -A^\tau \end{pmatrix} \Bigg|\,   B=-B^\tau, C=-C^\tau\right\}.
  \end{align}
where $A,B,C \in \glfrak(N)$. The elements $h_i=E_{i,i}-E_{n-i+1,n-i+1}$ for $i=1, \dots, N$ form a basis of the Cartan subalgebra $\hfrak_{2N}$. Let $\{\eps_i\,|\,i=1, \dots N\}$ be the dual basis of $\hfrak_{2N}^*$. The root system of $\sofrak(2N)$ is given by $\{\pm(\eps_i\pm \eps_j)\,|\,1\le i<j\le N\}$. We choose the simple roots
\begin{align}\label{eq:simpleRoots2N}
  \alpha_i=\begin{cases}
        \vep_i-\vep_{i+1} & \mbox{for $1\le i \le N-1$,}\\
        \vep_{N-1} + \vep_{N} & \mbox{for $i=N$.}
        \end{cases}
\end{align}
Corresponding Chevalley generators of $\sofrak(2N)$ are given by
\begin{align}\label{eq:Chevalley1N1}
 e_{\alpha_i}&= E_{i,i+1}{-}E_{n-i,n-i+1},& f_{\alpha_i}&=E_{i+1,i}{-}E_{n-i+1,n-i}, &  h_{\alpha_i}&=h_i{-}h_{i+1}
\end{align}
for $1\le i \le N-1$, and
\begin{align}\label{eq:Chevalley1N2}
 e_{\alpha_N}&{=} E_{N-1,N+1}{-}E_{N,N+2}, & f_{\alpha_N}&{=}E_{N+1,N-1} {-} E_{N+2,N},  &h_{\alpha_N}&{=}h_{N-1}{+}h_N.
\end{align}
    {\bf Case $n=2N+1$ odd:} In this case, we can write
    \begin{align}\label{eq:so2N+1}
  \sofrak(2N+1)= \glfrak_S(2N+1)=\left\{\begin{pmatrix} A & b & B \\ d & 0 & -b^\tau\\ C & -d^\tau & -A^\tau \end{pmatrix} \Bigg|\, B=-B^\tau, C=-C^\tau\right\}
    \end{align}  
 where $b\in \C^N$ is a column vector, $d\in\C^N$ is a row vector and $A,B,C$ are $N\times N$-matrices as before. As in the even case, the elements $h_i=E_{i,i}-E_{n-i+1,n-i+1}$ for $i=1, \dots, N$ form a basis of $\hfrak_{2N+1}$, and we still denote the dual basis by $\{\eps_i\,|\,i=1, \dots, N\}$. The root system of $\sofrak(2N+1)$ is given by $\{\pm \eps_\ell, \pm(\eps_i\pm \eps_j)\,|\,\ell=1, \dots, N,\, 1\le i<j\le N\}$. We choose the simple roots
\begin{align*}
  \gamma_i=\begin{cases}
              \vep_i-\vep_{i+1} & \mbox{for $1\le i \le N-1$,}\\
              \vep_N & \mbox{for $i=N$.}
           \end{cases}
\end{align*}
Corresponding Chevalley generators $e_{\gamma_i}, f_{\gamma_i}, h_{\gamma_i}$ of $\sofrak(2N+1)$ are given by \eqref{eq:Chevalley1N1} for $1\le i \le N-1$, and
\begin{align*}
 e_{\gamma_N}&= E_{N,N+1}{-}E_{N+1,N+2}, & f_{\gamma_N}&=2(E_{N+1,N} {-} E_{N+2,N+1}),  &h_{\gamma_N}&=2h_N.
\end{align*}
In the case $n=2N+1$, it helps to define
\begin{align}\label{eq:bi-def}
  b_i&=E_{i,N+1}{-}E_{N+1,2N-i+2}, & d_i&=E_{N+1,i}{-}E_{2N-i+2,N+1} & &\mbox{for $i=1,\dots,N$.}
\end{align}
With this notation, we have $e_{\gamma_N}=b_N$ and $f_{\gamma_N}=2d_N$. Moreover, note that
\begin{align}\label{eq:2Nin2N1}
  [f_{\gamma_i},b_i]&=b_{i+1},&  [e_{\gamma_{i}},b_{i+1}]&=b_i, &  [f_{\gamma_i},d_{i+1}]&=-d_i,&  [e_{\gamma_{i}},d_{i}]&=-d_{i+1}
\end{align}
for $1\le i \le N-1$. It is convenient to consider the Chevalley generators \eqref{eq:Chevalley1N2} of the copy of $\sofrak(2N)$ inside $\sofrak(2N+1)$.
We have
\begin{align}
  [f_{\alpha_N},b_{N-1}]&=-d_{N}, & [f_{\alpha_N},b_N]&=d_{N-1},\label{eq:bc1}  \\
  [e_{\alpha_N},d_{N-1}]&=b_{N}, & [e_{\alpha_N},d_N]&=-b_{N-1}.\label{eq:bc2}
\end{align}
The above formulas \eqref{eq:bc1}, \eqref{eq:bc2} and \eqref{eq:2Nin2N1} reflect the fact that the matrices $\{b_i, d_i\,|\,i=1,\dots,N\}$ form a basis of the vector representation of the Lie subalgebra $\sofrak(2N)\subset \sofrak(2N+1)$ under the adjoint action. This is also evident from the block matrix representation \eqref{eq:so2N+1} of $\sofrak(2N+1)$. 
\subsection{Realising $\sofrak(2N-1)\subset \sofrak(2N)$ as a symmetric pair}\label{sec:symPair}
In the realisations \eqref{eq:so2N}, \eqref{eq:so2N+1}, there is an immediate embedding of Lie algebras $\varphi:\sofrak(2N)\hookrightarrow \sofrak(2N+1)$ given by
\begin{align*}
   \varphi\left(\begin{pmatrix} A & B\\ C & -A^\tau\end{pmatrix}\right) = \begin{pmatrix} A & \bnull & B \\ \bnull^t & 0 & \bnull^t \\ C & \bnull &-A^\tau  \end{pmatrix}
\end{align*}
where $\bnull\in \C^N$ denotes the zero column vector. The embedding $\sofrak(2N-1)\subset \sofrak(2N)$ is slightly less immediate. In the following, we will realise $\sofrak(2N-1)$  for $N\ge 2$ as the fixed Lie subalgebra for a maximally split involution on $\sofrak(2N)$. Consider the $2N\times 2N$-matrix
\begin{align}\label{eq:J}
  J= \begin{pmatrix} \begin{array}{c|c|c} 1 & &  \\ \hline  & S_{2N-2} & \\ \hline & & 1  \end{array}\end{pmatrix}
\end{align}
where all empty blocks are filled with zero-matrices of the appropriate size.
As $J$ and the matrix $S_{2N}$ in \eqref{eq:S-def} commute, the map
\begin{align}\label{eq:theta-def}
  \theta: \sofrak(2N)\rightarrow \sofrak(2N), \qquad \theta(M)=-JM^tJ^{-1}
\end{align}
is well-defined. More explicitly, any $M\in \sofrak(2N)$ can be written as a block matrix
\begin{align}\label{eq:X}
  M=\begin{pmatrix} a & -s^\tau & 0  \\ t & P & s \\ 0 & -t^\tau & -a\end{pmatrix}
\end{align}
where $a \in \C$, $s,t\in\C^{2N-2}$ are column vectors and $P\in \sofrak(2N-2)$. A direct calculation shows that
\begin{align}\label{eq:thetaX}
  \theta(M) = \begin{pmatrix} -a & -t^\tau & 0 \\ s & P & t \\ 0 & -s^\tau & a
              \end{pmatrix}.
\end{align}
We consider the fixed Lie subalgebra
\begin{align*}
  \sofrak(2N)^\theta = \{M\in \sofrak(2N)\,|\, \theta(M)=M\} = \glfrak_S(2N)\cap \glfrak_J(2N).
\end{align*}
Comparing \eqref{eq:X} and \eqref{eq:thetaX}, we see that
\begin{align}\label{eq:so2Ntheta}
  \sofrak(2N)^\theta=\left\{\begin{pmatrix} 0 & -s^\tau & 0  \\ s & P & s \\ 0 & -s^\tau & 0 \end{pmatrix} \,\Bigg|\, P\in \sofrak(2N-2)\right\} 
\end{align}
where $s\in \C^{2N-2}$ is a column vector.
\begin{lem}\label{lem:eta}
  The map $\eta:\sofrak(2N-1)\rightarrow \sofrak(2N)$ defined by
\begin{align*}  
  \eta\left(\begin{pmatrix} A & b & B \\ d & 0 & -b^\tau\\ C & -d^\tau & -A^\tau \end{pmatrix} \right) =  \begin{pmatrix} 0 & \frac{1}{2}d & -b^\tau & 0  \\ b & A & 2B& b  \\ -\frac{1}{2}d^\tau & \frac{1}{2}C & - A^\tau& -\frac{1}{2}d^\tau \\ 0 &\frac{1}{2}d &  -b^\tau & 0\end{pmatrix}
\end{align*}
  is an injective homomorphism of Lie algebras with image $\sofrak(2N)^\theta$. 
\end{lem}  
\begin{proof}
  It follows from \eqref{eq:so2N+1} and \eqref{eq:so2Ntheta} that $\eta$ is an injective linear map with image  $\sofrak(2N)^\theta$. A direct calculation shows that the map $\eta$ respects the Lie brackets.
\end{proof}  
The above Lemma shows that $(\sofrak(2N),\sofrak(2N-1))$ is a symmetric pair.
We can identify the Chevalley generators of $\sofrak(2N-1)$ inside $\sofrak(2N)$ via the embedding $\eta$.
Observe first that
\begin{align}\label{eq:etaxi}
  \eta(e_{\gamma_i})&{=}e_{\alpha_{i+1}},&
  \eta(f_{\gamma_i})&{=}f_{\alpha_{i+1}},&
  \eta(h_{\gamma_i})&{=}h_{\alpha_{i+1}}
\end{align}
for $i=1, \dots, N-2$. Again, considering the Chevalley generators $e_{\alpha_{N-1}}, f_{\alpha_{N-1}}$ of $\sofrak(2N-2)$ inside $\sofrak(2N-1)$ we have
\begin{align}\label{eq:etaxN-1}
  \eta(e_{\alpha_{N-1}})&=2 e_{\alpha_N}, & \eta(f_{\alpha_{N-1}})&=\frac{1}{2} f_{\alpha_N}.
\end{align}  
  Next, recall the elements $b_i\in \sofrak(2N-1)$ defined by \eqref{eq:bi-def} for $i=1, \dots, N-1$. We have
\begin{align}\label{eq:f1+tf1}
  f_{\alpha_1} + \theta(f_{\alpha_1}) =  \eta(b_1)
\end{align}
and hence
\begin{align}
  \eta(e_{\gamma_{N-1}}) &= \eta(b_{N-1})\nonumber\\
  &\stackrel{\eqref{eq:2Nin2N1}}{=} \eta\big([f_{\gamma_{N-2}},[f_{\gamma_{N-3}},\dots,[f_{\gamma_1}, b_1]\dots]] \big)\nonumber\\
  &\stackrel{\eqref{eq:etaxi}}{=} [f_{\alpha_{N-1}},[f_{\alpha_{N-2}},\dots,[f_{\alpha_2}, \eta(b_1)]\dots]] \nonumber\\
   &\stackrel{\eqref{eq:f1+tf1}}{=} [f_{\alpha_{N-1}},[f_{\alpha_{N-2}},\dots,[f_{\alpha_2}, f_{\alpha_1} + \theta(f_{\alpha_1})]\dots]]   \label{eq:eN-1}
\end{align}  
\begin{align}
  \eta(f_{\gamma_{N-1}}) &= 2\eta(d_{N-1})\nonumber\\
  &\stackrel{\eqref{eq:bc1}}{=} -2 \eta\big([f_{\alpha_{N-1}},b_{N-2}]\big)\nonumber\\
  &\stackrel{\eqref{eq:2Nin2N1}}{=}- 2\eta\big([f_{\alpha_{N-1}},[f_{\gamma_{N-3}},\dots,[f_{\gamma_1}, b_1]\dots]] \big)\nonumber\\
  &\stackrel{\eqref{eq:etaxi},\eqref{eq:etaxN-1}}{=}-[f_{\alpha_{N}},[f_{\alpha_{N-2}},\dots,[f_{\alpha_2}, \eta(b_1)]\dots]] \nonumber\\
   &\stackrel{\eqref{eq:f1+tf1}}{=}- [f_{\alpha_{N}},[f_{\alpha_{N-2}},\dots,[f_{\alpha_2}, f_{\alpha_1} + \theta(f_{\alpha_1})]\dots]].  \label{eq:fN-1}
\end{align}
Moreover, we have
\begin{align}\label{eq:etaxi2}
  \eta(h_{\gamma_{N-1}})= \eta\big(2(E_{N-1,N-1}-E_{N+1,N+1})\big)=2h_N
   =h_{\alpha_N}- h_{\alpha_{N-1}}.
\end{align}  
We will use Equations \eqref{eq:eN-1} and \eqref{eq:fN-1} in Section \ref{sec:RV4Bc} to identify $q$-analogues of $e_{\gamma_{N-1}}$ and $f_{\gamma_{N-1}}$ inside the QSP-subalgebra $\cB_\bc$. 
\subsection{The Satake diagram of the symmetric pair $(\sofrak(2N),\sofrak(2N-1))$}
The involution $\theta$ of $\sofrak(2N)$ defined by \eqref{eq:theta-def} satisfies properties (1), (2), (3) in Section \ref{sec:symSatake} with $X=\{2,3,\dots,N\}$.
Indeed, by \eqref{eq:thetaX} we have
\begin{align}
  \theta(e_1)&=\theta(E_{1,2}-E_{N+2,N+1})=-E_{2N-1,1} + E_{2N,2} \nonumber\\
 &=(-1)^{N+1}[f_2,[f_3,{\dots}[f_{N-2},[f_{N},[f_{N-1},[f_{N-2},{\dots}[f_2,f_1]{\dots}]]]]{\dots}]],\label{eq:ms1}\\
  \theta(f_1)&=\theta(E_{2,1}-E_{2N,2N-1})=-E_{1,2N-1} + E_{2,2N} \nonumber\\  &=(-1)^{N+1}[e_2,[e_3,{\dots}[e_{N-2},[e_{N},[e_{N-1},[e_{N-2},{\dots}[e_2,e_1]{\dots}]]]]{\dots}]]\label{eq:ms2}
\end{align}
where to shorten notation we wrote $e_i, f_i$ instead of $e_{\alpha_i}$ and $f_{\alpha_i}$, respectively. Formulas \eqref{eq:ms1} and \eqref{eq:ms2} show that condition (3) in Section \ref{sec:symSatake} holds with $X=\{2,3,\dots,N\}$, and hence the involution $\theta$ defined by \eqref{eq:theta-def} is maximally split.
On the level of roots we obtain
\begin{align*} \Theta(-\alpha_1)-\alpha_1=2(\alpha_2+\dots+\alpha_{N-2})+\alpha_{N-1}+\alpha_N\in Q_X^+
\end{align*}  
as expected, and the involution $\theta$ corresponds to the Satake diagram of type $DII$ depicted in Figure \ref{fig:DII}. 
	
\section{Very non-standard quantum $\sofrak(2N-1)$}
The theory of quantum symmetric pairs provides a coideal subalgebra $\cB_\bc$ of $U_q(\sofrak(2N))$ corresponding to the Satake diagram in Figure \ref{fig:DII}. The algebra $\cB_\bc$ is a quantum group analogue of $U(\sofrak(2N-1))$. We recall the definition of $\cB_\bc$ in the conventions of \cite{a-Kolb14} and its presentation in terms of generators and relations as given in \cite{a-BalaKolb15}.
\subsection{Preliminaries on quantum groups}\label{sec:prelim}
For the remainder of this paper we fix $\gfrak=\sofrak(2N)$. Let $\qfield$ be a field 
and fix an element $q\in \qfield\setminus \{0\}$  which is not a root of unity. Let $\Uq=\uqg$ be the corresponding Drinfeld-Jimbo quantum enveloping algebra defined over $\qfield$ in terms of generators $E_i, F_i, K_i^{\pm 1}$ for $i=1, \dots, N$ and relations given in \cite[4.3]{b-Jantzen96}. For simplicity we write $I=\{1,2, \dots, N\}$. Let $U^+, U^-$ and $U^0$ denote the subalgebras of $\Uq$ generated by $\{E_i\,|\,i\in I\}$, $\{F_i\,|\,i\in I\}$ and $\{K_i^{\pm 1}\,|\,i\in I\}$, respectively. By \cite[4.21]{b-Jantzen96} the algebra $\Uq$ has a triangular decomposition in the sense that the multiplication map
\begin{align}\label{eq:triang}
  U^-\ot U^0\ot U^+\rightarrow \Uq
\end{align}
is a linear isomorphism. Recall the set of simple roots $\Pi=\{\alpha_i\,|\,i\in I\}$ given by \eqref{eq:simpleRoots2N}, the root lattice $Q=\Z\Pi$ and the positive cone $Q^+=\N_0\Pi$. For any $\mu\in Q$ let $U_\mu=\{u\in U\,|\,K_iuK_i^{-1}=q^{\mu(h_{\alpha_i})}u \mbox{ for all $i\in I$}\}$ be the corresponding weight space. We also write $U^\pm_\mu=U_\mu\cap U^\pm$.
Set $\Pi^\vee=\{h_{\alpha_i}\,|\,i\in I\}$ and let $Q^\vee=\Z\Pi^\vee$ be the coroot lattice. We can view $U^0$ as the group algebra of $Q^\vee$ and we write $K_h=\prod_{i\in I}K_i^{n_i}$ if $h=\sum_{i\in I}n_ih_{\alpha_i}$. With this notation we have
\begin{align}\label{eq:KEKF}
  K_h E_i K_h^{-1}=q^{\alpha_i(h)} E_i, \quad K_h F_i K_h^{-1}=q^{-\alpha_i(h)} F_i \qquad \mbox{for all $h\in Q^\vee, i\in I$.}
\end{align}  
Recall that $\Uq$ is a Hopf algebra with coproduct $\kow$, counit $\vep$ and antipode $S$ uniquely determined by
\begin{align*}
  \kow (E_i)=E_i\ot 1 + K_i\ot E_i, \quad \kow(F_i)=F_i\ot K_i^{-1}+1\ot F_i,\quad \kow(K_i)=K_i\ot K_i.
\end{align*}  
Let $T_i:\Uq\rightarrow \Uq$ for $i\in I$ denote the Lusztig braid group automorphisms on $\Uq$ in the conventions of \cite[8.14]{b-Jantzen96}. In particular, for $i,j\in I$ with $a_{ij}=-1$ we have
\begin{align}\label{eq:ij-ji}
  T_i(E_j)=[E_i,E_j]_{q^{-1}}=T_j^{-1}(E_i), \qquad T_{i}(F_j)=[F_j,F_i]_q=T_j^{-1}(F_i).
\end{align}
Let $W$ be the Weyl group of $\gfrak$ generated by the simple reflections $s_i$ corresponding to the simple roots $\alpha_i$ for $i\in I$. For any $w\in W$ with reduced expression $w=s_{i_1} \dots s_{i_\ell}$ let $T_w=T_{i_1}\dots T_{i_\ell}$ denote the corresponding Lusztig automorphism of $\Uq$.

Finally, let $\sigma:\Uq\rightarrow \Uq$ denote the involutive algebra antiautomorphism defined by $\sigma(E_i)=E_i$, $\sigma(F_i)=F_i$ and $\sigma(K_i)=K_i^{-1}$, see \cite[3.1.3]{b-Lusztig94}.

\subsection{A very non-standard deformation of $U(\sofrak(2n-1))$}
Let $X\subset I$ be the subset $X=\{2,\dots, N\}$ and let $\cM_X\subset \Uq$ be the subalgebra generated by all $E_j, F_j, K_j^{\pm 1}$ for $j\in X$. By construction, the Hopf-subalgebra $\cM_X$ is isomorphic to the Drinfeld-Jimbo quantum enveloping algebra $U_q(\sofrak(2N-2))$. Let $W_X\subset W$ be the parabolic subgroup corresponding to the subset $X$ and let $w_X\in W_X$ be the longest element.
Let $\cB_\bc\subset \Uq$ be the subalgebra generated by $\cM_X$ and the element
\begin{align*}
  B_1= F_1 - c_1 T_{w_X}(E_1) K_1^{-1}.
\end{align*}
The algebra $\cB_\bc$ depends on a single parameter $c_1\in \qfield\setminus \{0\}$.
We call $\cB_\bc$ the very-nonstandard quantum deformation of $U(\sofrak(2N-1))$. The pair $(\Uq,\cB_\bc)$ is an example of a quantum symmetric pair as constructed by G.~Letzter in \cite{a-Letzter99a}. Here we follow the conventions of \cite{a-Kolb14}. In particular, $\cB_\bc\subset \Uq$ is a right coideal subalgebra, that is $\kow(\cB_\bc)\subset \Uq\ot \cB_\bc$. As in \cite{a-Kolb14} we also refer to the algebra $\cB_\bc$ as a quantum symmetric pair (QSP) coideal subalgebra of $\Uq$.
\begin{rema}
  The algebra $\cB_\bc$ was first defined by Noumi and Sugitani in \cite[Section 3]{a-NS95} as case `$BDI$ for $\ell=1$ and even $N$'. Their construction is more in the spirit of the FRT approach to quantum groups and relies on an explicit solution of the reflection equation which is a $q$-analogue of the matrix $J$ in \eqref{eq:J}. Noumi and Sugitani denoted their QSP-coideal subalgebra by $U^{\mathrm{tw}}_q(\kfrak)$. Letzter showed in \cite[Section 6]{a-Letzter99a} that $\cB_\bc$ as defined here and $U^{\mathrm{tw}}_q(\kfrak)$ as defined in \cite{a-NS95} coincide for a suitable choice of the parameter $c_1$. 
\end{rema}  
For $i=1, \dots, N-1$ define reduced words $\sigma_i$ in the Weyl group $W$ by
\begin{align}\label{eq:sigmai}
  \sigma_i=s_i\dots s_{N-2} s_{N-1} s_{N} s_{N-2}\dots s_i. 
\end{align}
Then $w_X=\sigma_2\dots\sigma_{N-1}$ is a reduced expression for the longest element $w_X\in W_X$. Using Equation \eqref{eq:ij-ji} we obtain
\begin{align}
  T_{w_X}&(E_1)=T_{\sigma_2}(E_1)=T_2\dots T_{N-2} T_{N-1} T_{N} T_{N-2} \dots T_2(E_1)\nonumber\\
  &=\big[[E_2,\mydots[E_{N-2},E_N]_{q^{-1}}\mydots]_{q^{-1}},[E_{N-1},[E_{N-2}, \mydots[E_2,E_1]_{q^{-1}}\mydots]_{q^{-1}}]_{q^{-1}} \big]_{q^{-1}}.\label{eq:TwXE1}
\end{align}
This expression allows us to show that the algebras $\cB_\bc$ for different parameters $\bc=(c_1)$ are isomorphic via Hopf algebra automorphisms of $\Uq$.
\begin{prop}\label{prop:BcBd-iso}
  Let $c_1, d_1 \in \qfield\setminus\{0\}$ and $\cB_\bc$, $\cB_\bd$ the corresponding very nonstandard quantum deformations of $U(\sofrak(2N-1))$. Then there exists a Hopf algebra automorphism $\phi:\Uq\rightarrow \Uq$ such that $\phi(\cB_\bc)=\cB_\bd$.
\end{prop}
\begin{proof}
  Define the automorphism $\phi:\Uq\rightarrow \Uq$ by $\phi(K)=K$ for all $K\in U^0$ and
  \begin{align*}
    \phi(E_i)=\begin{cases}
                E_i & \mbox{if $i\neq N$,}\\
                \frac{d_1}{c_1} E_N & \mbox{if $i=N$,}
              \end{cases} \qquad 
    \phi(F_i)=\begin{cases}
                F_i & \mbox{if $i\neq N$,}\\
                \frac{c_1}{d_1} F_N & \mbox{if $i=N$.}
              \end{cases}
  \end{align*}
  By Equation \eqref{eq:TwXE1} we have
  \begin{align*}
     \phi(F_1- c_1 T_{w_X}(E_1)K_1^{-1}) = F_1 - d_1 T_{w_X}(E_1)K_1^{-1}.
  \end{align*}
  Moreover, $\phi$ restricts to a Hopf algebra automorphism of $\cM_X$. This shows that $\phi(\cB_\bc)=\cB_\bd$.
\end{proof}  
\begin{rema}
  In \cite{a-Letzter99a} and \cite{a-Kolb14} the quantised enveloping algebra $\uqg$ and its coideal subalgebra $\cB_\bc$ are defined over the field $\qfield=\field(q)$ of rational functions in an indeterminate $q$ and $\field$ is assumed to be of characteristic zero. However, all relevant results in \cite{a-Kolb14} not involving specialisation also hold in the more general setting that $q\in \qfield$ is not a root of unity in an arbitrary field $\qfield$. Only once we use specialisation in Section \ref{sec:specialization} will we have to restrict to fields $\qfield$ of characteristic $0$ and $q\in \qfield$ which is transcendental over $\Q$. This is similar to the approach in \cite[Section 5]{b-Jantzen96} which we are trying to emulate.
\end{rema}  
\subsection{Generators and relations for $\cB_\bc$}\label{sec:gen-rels}
Recall the Lusztig-Kashiwara skew derivation ${}_ir, r_i:U^+\rightarrow U^+$ defined by
\begin{align}\label{eq:uFFu}
  u F_i - F_i u = \frac{r_i(u)K_i - K_i^{-1}{}_ir(u)}{q-q^{-1}} \qquad \mbox{for all $u\in U^+$,}
\end{align}
see \cite[Lemma 6.17]{b-Jantzen96}, and define
\begin{align}\label{eq:Zidef}
  \cZ_1 = r_1(T_{w_X}(E_1)).
\end{align}
For $N=3$ one gets $\cZ_1=(1-q^{-2})^2E_2E_3$. For $N\ge 4$ one calculates as in \cite[Section 2.2]{a-BalaKolb15} to obtain
\begin{align}\label{eq:Z1-explicit}
  \cZ_1=(1-q^{-2})[T_{3}^{-1}\dots T_{N-2}^{-1}T_N^{-1}T_{N-1}^{-1}T_{N-2}^{-1}\dots T_3^{-1}(E_2),E_2]_{q^{-2}}.
\end{align}
For $N\ge 4$ Equation \eqref{eq:Zidef} implies that $T_i(\cZ_1)=\cZ_1$ if $i\ge 3$. Hence Equation \eqref{eq:Z1-explicit} can be rewritten as
\begin{align}\label{eq:Z1-explicit-T}
  \cZ_1=(1-q^{-2})[E_2,T_{3}\dots T_{N-2}T_NT_{N-1}T_{N-2}\dots T_3(E_2)]_{q^{-2}}.
\end{align}
By \cite[8.26]{b-Jantzen96} one has
\begin{align}
  T_i(U^+)\cap U^+=\{x\in U^+\,|\,r_i(x)=0\},\label{eq:rinull}\\
  T_i^{-1}(U^+)\cap U^+=\{x\in U^+\,|\,{}_ir(x)=0\}.\label{eq:irnull}
\end{align}  
By application of \eqref{eq:rinull} and \eqref{eq:irnull} to the expressions \eqref{eq:Z1-explicit} and \eqref{eq:Z1-explicit-T}, respectively, we hence obtain
\begin{align*}
  r_2(\cZ_1)&=(1-q^{-2})^2T_{3}^{-1}\dots T_{N-2}^{-1}T_N^{-1}T_{N-1}^{-1}T_{N-2}^{-1}\dots T_3^{-1}(E_2),\\
  {}_2r(\cZ_1)&= (1-q^{-2})^2T_{3}\dots T_{N-2}T_NT_{N-1}T_{N-2}\dots T_3(E_2)
\end{align*}  
in the case $N\ge 4$. With these preparations we are ready to describe the defining relations of the algebra $\cB_\bc$ explicitly.
\begin{thm}[\upshape{\cite[Theorem 7.1]{a-Letzter03}, see also \cite[Theorem 3.9]{a-BalaKolb15}}] The algebra $\cB_\bc$ is generated over $\cM_X$ by the element $B_1$ subject only to the relations
\begin{align*}
  K_i B_1 K_i^{-1} &= q^{-(\alpha_i,\alpha_1)} B_1 & &\mbox{for $i=2, \dots, N$,}\\
  E_i B_1 &= B_1 E_i& &\mbox{for $i=2, \dots, N$,}\\
  F_j B_1 &= B_1 F_j& &\mbox{for $j=3, \dots, N$,}
\end{align*}
\begin{align}
  B_1^2 F_2 - (q+q^{-1}) B_1 F_2 B_1 + F_2 B_1^2 & =  \label{eq:qSerre-Bc1}\\
       -qc_1 \Big(F_2 \cZ_1 + &
  \frac{q r_2(\cZ_1)K_2 + q^{-1} {}_2r(\cZ_1)K_2^{-1}}{(q-q^{-1})^2} \Big),\nonumber\\
   F_2^2 B_1 - (q + q^{-1}) F_2 B_1 F_2 + B_1 F_2^2 &=0.
\end{align}  
\end{thm}
The algebra $\cB_\bc$ is invariant under the Lusztig automorphisms $T_i$ for $i=2, \dots, N$. More explicitly, a direct calculation shows that
\begin{align}
  T_i(B_1)&=\begin{cases} B_1 & \mbox{if $i=3, \dots, N$,}\\
                        [B_1,F_2]_q & \mbox{if $i=2$,}
           \end{cases}\label{eq:TiB}\\
  T_i^{-1}(B_1)&=\begin{cases}B_1 & \mbox{if $i=3, \dots, N$,}\\
  [F_2,B_1]_q & \mbox{if $i=2$,}
           \end{cases}\label{eq:Ti-B}
\end{align}  
see also \cite[Theorem 4.2]{a-BaoWang18b}.

Recall the algebra antiautomorphism $\sigma$ from Section \ref{sec:prelim}. There exists an involutive algebra antiautomorphism $\sigma^B:\cB_\bc\rightarrow \cB_\bc$ defined by
\begin{align}\label{eq:sigmaB-def}
  \sigma^B|_{\cM_X}=\sigma|_{\cM_X} \quad \mbox{and} \quad \sigma^B(B_1)=B_1.
\end{align}
To show that $\sigma^B$ is well defined, it suffices to check that the right hand side of Equation \eqref{eq:qSerre-Bc1} is preserved under $\sigma$. This follows from the relation $\sigma(\cZ_1)=\cZ_1$, which was proved in \cite[Proposition 2.3]{a-BalaKolb15}, and the relation $\sigma(r_2(\cZ_1))={}_2r(\cZ_1)$ which is a direct consequence of the definition of $r_i$ and ${}_ir$ via Equation \eqref{eq:uFFu}. Note also that $K_2 \,{}_2r(\cZ_1)={}_2r(\cZ_1)K_2$ as ${}_2r(\cZ_1)\in (\cM^+_X)_{w_X(\alpha_1)-\alpha_1-\alpha_2}$.
\subsection{The standard filtration on $\cB_\bc$}\label{sec:standardFilt}
Let $\cA$ be the subalgebra of $\Uq$ generated by $\cM_X$ and the element $F_1$. Let $\cM_X^+, \cM_X^-$ and $U^0_X$ be the subalgebras of $\cM_X$ generated by $\{E_j\,|\,j\in X\}, \{F_j\,|\,j\in X\}$ and $\{K_j^{\pm 1}\,|\,j\in X\}$, respectively. By the triangular decomposition of $\Uq$, the multiplication map
\begin{align}\label{eq:Atriang}
  U^-\ot U^0_X\cM_X^+\rightarrow \cA
\end{align}
is a linear isomorphism. The algebra $\cA$ is $\N_0$-graded by the degree function $\deg$ given by $\deg(F_1)=1$ and $\deg(m)=0$ for all $m\in \cM_X$. The algebra $\cB_\bc$, on the other hand, has an $\N_0$-filtration $\cF_\ast$ defined by the degree function on the generators given by
\begin{align*}
  \deg(B_1)&=1, & \deg(m)=0 \quad \mbox{for all $m\in \cM_X$.}
\end{align*}
Let $\gr(\cB_\bc)$ denote the associated graded algebra. The defining relations of $\cB_\bc$ in Section \ref{sec:gen-rels} imply that there is a surjective homomorphism $\varphi:\cA\rightarrow \gr(\cB_\bc)$ of graded algebras such that
\begin{align*}
  \varphi(F_1)&=B_1\in \cF_1(\cB_\bc)/\cF_0(\cB_\bc),& \varphi(m)=m\in \cF_0(\cB_\bc) \quad \mbox{for $m\in \cM_X$.}
\end{align*}
It follows from \cite[Proposition 6.2]{a-Kolb14} that the map $\varphi:\cA\rightarrow \gr(\cB_\bc)$ is an isomorphism of graded algebras.

By Equations \eqref{eq:TiB}, \eqref{eq:Ti-B} the algebra isomorphisms $T_i:\cB_\bc\rightarrow \cB_\bc$ for $i\in X$ restrict to linear isomorphisms $T_i:\cF_n(\cB_\bc)\rightarrow \cF_n(\cB_\bc)$ for any $n\in \N_0$. Hence the maps $T_i:\cB_\bc\rightarrow \cB_\bc$ induce algebra isomorphisms $\gr(T_i):\gr(\cB_\bc)\rightarrow \gr(\cB_\bc)$. The second formula in \eqref{eq:ij-ji} and Equations \eqref{eq:TiB}, \eqref{eq:Ti-B} now imply the following result.
\begin{lem}\label{lem:grT}
  For any $i\in X$ the algebra isomorphism $\varphi:\cA\rightarrow \gr(\cB_\bc)$ is compatible with the Lusztig automorphism $T_i$ in the sense that $\varphi\circ T_i = \gr(T_i)\circ \varphi$.
\end{lem}  
\section{A Poincar\'e-Birkhoff-Witt Theorem for $\cB_\bc$}\label{sec:PBW}
We now use our understanding of the classical symmetric pair $(\sofrak(2N),\sofrak(2N-1))$ from Section \ref{sec:sym-pair} to define root vectors for the algebra $\cB_\bc$. We will see in Theorem \ref{thm:PBW-Bc} that the ordered monomials in the root vectors provide a Poincar\'e-Birkhoff-Witt basis for $\cB_\bc$. This statement can be reformulated as a triangular decomposition for $\cB_\bc$ in Corollary \ref{cor:Bc-triang}. We end the section with some preliminary results on $q$-commutators of root vectors.
\subsection{The Poincar\'e-Birkhoff-Witt Theorem for $U_q(\sofrak(2N))$}
Recall the reduced words $\sigma_i$ for $i=1, \dots, N-1$ defined by \eqref{eq:sigmai}. The word
\begin{align}\label{eq:reduced}
  w_0=\sigma_1 \sigma_2 \dots \sigma_{N-1}
\end{align}
is a reduced expression of the longest element in $W$. Write the reduced expression \eqref{eq:reduced} as $w_0=s_{i_1}\dots s_{i_{N(N-1)}}$. We obtain the $N(N-1)$ positive roots of $\sofrak(2N)$ as
\begin{align*}
  \beta_j=s_{i_1} \dots s_{i_{j-1}}(\alpha_{i_j}) \qquad \mbox{for $j=1, \dots, N(N-1)$.}
\end{align*}  
Let $\tau:I\rightarrow I$ be the nontrivial diagram automorphism, i.e.~$\tau(i)=i$ for $i=1, \dots, N-2$ and $\tau(N-1)=N$, $\tau(N)=N-1$. As $\sigma_i(\alpha_j)=\alpha_{\tau(j)}$ for $i<j$, we see that $\alpha_1$ is contained with a nonzero coefficient only in the roots $\beta_1,\dots, \beta_{2(N-1)}$ but not in $\beta_j$ for $j>2(N-1)$.
Define root vectors $F_{\beta_j}\in U^-$ by
\begin{align}\label{eq:Fbetaj-def}
  F_{\beta_j}= T_{i_1} \dots T_{i_{j-1}}(F_{i_j}) \qquad \mbox{for $j=1, \dots, N(N-1)$.}
\end{align}
By the Poincar\'e-Birkhoff-Witt Theorem \cite[8.24]{b-Jantzen96} the ordered monomials
\begin{align}\label{eq:FJ-def}
  F_{\cJ}=F_{\beta_{N(N-1)}}^{j_{N(N-1)}} \dots  F_{\beta_2}^{j_2}F_{\beta_1}^{j_1} \qquad \mbox{ for $\cJ=(j_1, \dots, j_{N(N-1)})\in \N_0^{N(N-1)}$}
\end{align}
form a $\qfield$-vector space basis $U^-$. We want to rewrite the root vectors  $F_{\beta_j}$ in the case where $\beta_j$ contains $\alpha_1$ with a non-zero coefficient.
\begin{lem}
  For $j\in\{1,\dots, 2(N-1)\}$ the root vector $F_{\beta_j}$ can be written as
  \begin{align}\label{eq:Fbetaj}
    F_{\beta_j}=\begin{cases}
      F_1 & \mbox{if $j=1$,}\\
      T_j^{-1}\dots T_{2}^{-1}(F_1) & \mbox{if $2\le j\le N-1$,}\\
      T_{2N-j}^{-1}\dots T_N^{-1}T_{N-2}^{-1}\dots T_2^{-1}(F_1 ) & \mbox{if $N\le j\le 2(N-1)$.}
              \end{cases}
  \end{align}  
\end{lem}  
\begin{proof}
  It follows from \eqref{eq:ij-ji} that Equation \eqref{eq:Fbetaj} holds for $j=1, \dots, N$. For $j\ge N+1$ we have
  \begin{align*}
    F_{\beta_j}=T_1 \dots  T_N T_{N-2}\dots T_{k+1}(F_k) \qquad \mbox{where $k=2N-j-1$.} 
  \end{align*}
  Using again  \eqref{eq:ij-ji} we hence obtain
  \begin{align*}
    F_{\beta_j}= T_{1}\dots T_N T^{-1}_{k} \dots T_{N-3}^{-1}(F_{N-2}).
  \end{align*}
  As $T_{1} \dots T_{N-2} T^{-1}_k \dots T^{-1}_{N-3}= T^{-1}_{k+1} \dots T_{N-2}^{-1} T_1 \dots T_{N-2}$ we obtain
  \begin{align*}
    F_{\beta_j}=  T^{-1}_{k+1} \dots T_{N-2}^{-1} T_1 \dots  T_N(F_{N-2}).
  \end{align*}
  Now we use $T_{N-2}T_{N-1}T_N(F_{N-2})=T_{N-1}^{-1}T_N^{-1}(F_{N-2})$ and \eqref{eq:ij-ji} to obtain
  \begin{align*}
    F_{\beta_j}= T^{-1}_{2N-j} \dots T^{-1}_{N-2} T^{-1}_N \dots T^{-1}_2(F_1) 
  \end{align*}
  where we also substituted $k+1=2N-j$. This settles the remaining cases. 
  \end{proof}  

\subsection{Root vectors for $\cB_\bc$}\label{sec:RV4Bc}
Recall from Equation \eqref{eq:TiB}, \eqref{eq:Ti-B} that the Lusztig automorphisms $T_i$ for $i\in X$ leave $\cB_\bc$ invariant. In analogy to \eqref{eq:Fbetaj} we define root vectors of $\cB_\bc$ for $j=1, \dots, N(N-1)$ by
\begin{align}\label{eq:Bbetaj}
    B_{\beta_j}=\begin{cases}
      B_1 & \mbox{if $j=1$,}\\
      T_j^{-1}\dots T_{2}^{-1}(B_1) & \mbox{if $2\le j\le N-1$,}\\
      T_{2N-j}^{-1}\dots T_N^{-1}T_{N-2}^{-1}\dots T_2^{-1}(B_1 ) & \mbox{if $N\le j\le 2(N-1)$.}\\
      F_{\beta_j} & \mbox{if $j >2(N-1)$.}
               \end{cases}
  \end{align}  
In view of Equation \eqref{eq:Ti-B} we can write the root vectors for $\cB_\bc$ explicitly as iterated $q$-commutators.
\begin{lem}\label{lem:Bbetaj2}
  For $j\in \{1,\dots,2(N-1)\}$ the root vectors $B_{\beta_j}$ can be written as
  \begin{align}\label{eq:Bbetaj2}
     B_{\beta_j}=\begin{cases}
      B_1 & \mbox{if $j=1$,}\\
      [F_j,[F_{j-1},\dots,[F_3,[F_2,B_1]_q]_q{\dots}]_q]_q & \mbox{if $2\le j\le N-1$,}\\
     [F_N,[F_{N-2},\dots,[F_3,[F_2,B_1]_q]_q{\dots}]_q]_q & \mbox{if $j=N$,}
      \\  
    [F_N,[F_{N-1},\dots,[F_3,[F_2,B_1]_q]_q{\dots}]_q]_q & \mbox{if $j=N+1$,}
      \\
      [T_{N}{\dots}T_{2N-j+1}(F_{2N-j}),T_{N-2}^{-1}{\dots} T_2^{-1}(B_1 )]_q & \mbox{if $N{+}2{\le} j{\le} 2(N{-}1)$.}
      \end{cases}
  \end{align}  
\end{lem}
\begin{proof}
  For $2\le j \le N+1$ Equation \eqref{eq:Bbetaj2} follows from Equations \eqref{eq:ij-ji}, \eqref{eq:Ti-B}. In the case $N+2\le j \le 2(N-1)$ we again write $k=2N-j-1$ and get
  \begin{align*}
    B_{\beta_j}&= T_{k+1}^{-1} \dots T_{N-2}^{-1} T_N^{-1} \dots T_2^{-1}(B_1)\\
    &=T_{k+1}^{-1}\dots T_{N-2}^{-1}\big([T_N^{-1}T_{N-1}^{-1}(F_{N-2}),T^{-1}_{N-3}\dots T_{2}^{-1}(B_1)]_q\big)\\
    &=T_{k+1}^{-1}\dots T_{N-3}^{-1}\big([T_N T_{N-1} (F_{N-2}),T^{-1}_{N-2}\dots T_{2}^{-1}(B_1)]_q\big)\\
    &=[T_N \dots T_{k+2} (F_{k+1}),T^{-1}_{N-2}\dots T_{2}^{-1}(B_1)]_q
  \end{align*}
  where we made repeated use of Equation \eqref{eq:ij-ji}.
\end{proof}
\begin{rema}\label{rem:q=1}
  Lemma \ref{lem:Bbetaj2} allows us to informally identify the limit of the root vectors $B_{\beta_j}$ for $q \rightarrow 1$. Recall the notation from Sections \ref{sec:rootChevalley} and \ref{sec:symPair}. By \eqref{eq:f1+tf1} the generator $B_1\in \cB_\bc$ is a $q$-analogue of the root vector $b_1\in \sofrak(2N-1)$ corresponding to the root $\vep_1$. Similarly, the generators $F_i\in \cB_\bc$ for $i\in X$ are $q$-analogues of the Chevalley generators $f_{\alpha_i}\in \sofrak(2N)$ and hence of $f_{\gamma_{i-1}}\in \sofrak(2N-1)$ if $i\le N-1$. Hence the first equation in \eqref{eq:2Nin2N1} shows that $B_{\beta_j}$ for $1\le j\le N-1$ are $q$-analogues of the elements $b_j\in \sofrak(2N-1)$. In particular,
\begin{align*}
  B_{\beta_{N-1}}=T^{-1}_{N-1}T^{-1}_{N-2}\dots T^{-1}_{2}(B_1)
\end{align*}
is a $q$-analogue of the Chevalley generator $e_{\gamma_{N-1}}\in \sofrak(2N-1)$, see also Equation \eqref{eq:eN-1}. Similarly, by Equation \eqref{eq:fN-1},
\begin{align*}
   B_{\beta_{N}}=T^{-1}_{N}T^{-1}_{N-2}\dots T^{-1}_{2}(B_1)
\end{align*}
is a $q$-analogue of $-f_{\gamma_{N-1}}\in \sofrak(2N-1)$.
We will make these argument precise in Section \ref{sec:Bc-special} using specialisation.
\end{rema}
The interpretation of $B_{\beta_{N-1}}$ and $-B_{\beta_N}$ as $q$-analogues of the Chevalley generators corresponding to the simple roots $\pm\gamma_{N-1}$ is confirmed by the following Lemma which also identifies $q$-analogues of the remaining simple root vectors. Let $Q^\vee_{2N-1}=\Z\{h_{\gamma_j}\,|\,j=1, \dots, N-1\}$ be the coroot lattice of $\sofrak(2N-1)$. By Equations \eqref{eq:etaxi} and \eqref{eq:etaxi2} the embedding $\eta$ induces a group homomorphism $\eta:Q^\vee_{2N-1}\rightarrow Q^\vee$.
\begin{lem}\label{lem:root-vectors}
  For $i=2\dots, N-1$ and any $h\in Q^\vee_{2N-1}$ the following relations hold:
  \begin{align}
    K_{\eta(h)} E_i K_{\eta(h)}^{-1} &= q^{\gamma_{i-1}(h)}E_i,&
    K_{\eta(h)} B_{\beta_{N-1}} K_{\eta(h)}^{-1}&= q^{\gamma_{N-1}(h)} B_{\beta_{N-1}}, \label{eq:root-vectors1}\\
    K_{\eta(h)} F_i K_{\eta(h)}^{-1} &= q^{-\gamma_{i-1}(h)} F_i,&
    K_{\eta(h)} B_{\beta_{N}} K_{\eta(h)}^{-1}&= q^{-\gamma_{N-1}(h)} B_{\beta_{N}}\label{eq:root-vectors2}.
  \end{align}
  Moreover, we have
  \begin{align}
     K_{\eta(h)} E_N K_{\eta(h)}^{-1} &= q^{(\gamma_{N-2}+2\gamma_{N-1})(h)}E_i,\label{eq:EN-root-vector}\\
    K_{\eta(h)} F_N K_{\eta(h)}^{-1}&= q^{-(\gamma_{N-2}+2\gamma_{N-1})(h)} F_N.\label{eq:FN-root-vector}
  \end{align}  
\end{lem}  
\begin{proof}
  We check the relations \eqref{eq:root-vectors1}, the relations \eqref{eq:root-vectors2} follow analogously. It suffices to verify the relations \eqref{eq:root-vectors1} for $h=h_{\gamma_j}$ with $j=1,\dots,N-1$. For $j=1, \dots, N-2$ we have
  \begin{align*}
    K_{\eta(h_{\gamma_j})} E_i K_{\eta(h_{\gamma_j})}^{-1} =q^{\alpha_i(\eta(h_{\gamma_j}))} E_i
    =q^{\alpha_i(h_{\alpha_{j+1}})} E_i =q^{\gamma_{i-1}(h_{\gamma_{j}})} E_i. 
  \end{align*}  
  Moreover, we have
  \begin{align*}
    K_{\eta(h_{\gamma_{N-1}})} E_i K_{\eta(h_{\gamma_{N-1}})}^{-1} =q^{\alpha_i(\eta(h_{\gamma_{N-1}}))} E_i
    =q^{\alpha_i(h_{\alpha_N}-h_{\alpha_{N-1}})} E_i= q^{\gamma_{i-1}(h_{\gamma_{N-1}})} E_i. 
  \end{align*}  
  This proves the first equation in \eqref{eq:root-vectors1}.
  Next we use the definition \eqref{eq:Bbetaj} of $B_{\beta_{N-1}}$ to calculate
   \begin{align}
     K_{\eta(h)} B_{\beta_{N-1}} K_{\eta(h)}^{-1} &=q^{-s_{N-1} s_{N-2}\dots s_{2}(\alpha_1)(\eta(h))} B_{\beta_{N-1}}\nonumber \\
     &= q^{-\alpha_1( s_2 \dots s_{N-2}s_{N-1}(\eta(h)))} B_{\beta_{N-1}} \label{eq:step}. 
   \end{align}
   For $j=1, \dots, N-1$ we have
   \begin{align*}
     -\alpha_1( s_2 \dots s_{N-2}s_{N-1}(\eta(h_{\gamma_{j}})))&=
     \begin{cases}
       2 & \mbox{if $j=N-1$}\\ -1 & \mbox{if $j=N-2$}\\  0 & \mbox{else} 
     \end{cases}\\
     &= \gamma_{N-1}(h_{\gamma_{j}}).
   \end{align*}
   Inserting the above Equation into \eqref{eq:step} we obtain the second equation in \eqref{eq:root-vectors1}.

   To verify Equation \eqref{eq:EN-root-vector} we consider $j=1, \dots, N-2$ and calculate
   \begin{align}
      K_{\eta(h_{\gamma_j})} E_N K_{\eta(h_{\gamma_j})}^{-1} =q^{\alpha_N(\eta(h_{\gamma_j}))} E_N
      =q^{\alpha_N(h_{\alpha_{j+1}})} E_i =q^{-\delta_{j,N-3}} E_N,\label{eq:EN1}\\
      K_{\eta(h_{\gamma_{N-1}})} E_N K_{\eta(h_{\gamma_{N-1}})}^{-1} =q^{\alpha_N(\eta(h_{\gamma_{N-1}}))} E_N
      =q^{\alpha_N(h_{\alpha_{N}}-h_{\alpha_{N-1}} )} E_N =q^{2} E_N. \label{eq:EN2}
   \end{align}
   On the other hand, as $\gamma_{N-2}+2\gamma_{N-1}=\epsilon_{N-2}+\epsilon_{N-1}$ we obtain
   \begin{align*}
     (\gamma_{N-2}+2\gamma_{N-1})(h_{\gamma_j})=\begin{cases}
     -1 & \mbox{if $j=N-3$,}\\
     2  & \mbox{if $j=N-1$,}\\
     0 & \mbox{else.}
     \end{cases}
   \end{align*}
   Comparison of the above with \eqref{eq:EN1} and \eqref{eq:EN2} proves Equation \eqref{eq:EN-root-vector}. Equation \eqref{eq:FN-root-vector} now follows from \eqref{eq:KEKF}.
\end{proof}  
\subsection{The PBW-basis for $\cB_\bc$}
Recall the definition \eqref{eq:FJ-def} of the PBW-monomials $F_\cJ\in U^-$ for $\cJ\in \N_0^{N(N-1)}$. Similarly, we use the root vectors \eqref{eq:Bbetaj} to define PBW-monomials for $\cB_\bc$ by
\begin{align}\label{eq:BJ-def}
  B_\cJ=B_{\beta_{N(N-1)}}^{j_{N(N-1)}} \dots  B_{\beta_2}^{j_2}B_{\beta_1}^{j_1} \qquad \mbox{ for $\cJ=(j_1, \dots, j_{N(N-1)})\in \N_0^{N(N-1)}$.}
\end{align}  
Recall the algebra isomorphism $\varphi:\cA\rightarrow \gr(\cB_\bc)$ from Section \ref{sec:standardFilt}. Lemma \ref{lem:grT} and Equations \eqref{eq:Fbetaj}, \eqref{eq:Bbetaj} imply that $\varphi(F_{\beta_j})=B_{\beta_j}$. This in turn implies that for any $\cJ=(j_1,\dots,j_{N(N-1)})\in \N_0^{N(N-1)}$ we have
\begin{align}\label{eq:varphiFJ}
  \varphi(\cF_\cJ)=B_\cJ\in \cF_j(\cB_\bc)/\cF_{j-1}(\cB_\bc) \qquad \mbox{with $j=\sum_{\ell=1}^{2N-2}j_\ell$.}
\end{align}
This observation is the main ingredient in the proof of the following Proposition.
\begin{prop}\label{prop:Bc-basis}
  The set $\{B_\cJ\,|\,\cJ\in \N_0^{N(N-1)}\}$ is a basis of $B_\bc$ as a right (or left) $U^0_X\cM_X^+$-module. 
\end{prop}
\begin{proof}
  By the PBW Theorem for $U^-$ and the triangular decomposition \eqref{eq:Atriang}, the algebra $\cA$ is a free right $U^0_X\cM_X^+$-module with basis $\{F_\cJ\,|\,\cJ\in \N_0^{N(N-1)}\}$ given by \eqref{eq:FJ-def}. Lemma \ref{lem:grT} and Equation \eqref{eq:varphiFJ} imply that the set
  \begin{align*}
    \{B_\cJ\,|\,\cJ=(j_1, \dots,j_{N(N-1)})\in \N_0^{N(N-1)}, \sum_{\ell=1}^{2N-2}j_\ell\le j\}
  \end{align*}
forms a basis of $\cF_j(\cB_\bc)$ as a right $U^0_X\cM_X^+$-module. This proves the proposition considering $\cB_\bc$ as a right $U^0_X\cM_X^+$-module. The corresponding statement about $\cB_\bc$ as a left $U^0_X\cM_X^+$-module is obtained by swapping tensor factors in the triangular decomposition \eqref{eq:Atriang}.  
\end{proof}
In analogy to \eqref{eq:Fbetaj-def} we define positive root vectors $E_{\beta_j}\in U^+$ by
\begin{align*}
  E_{\beta_j}= T_{i_1} \dots T_{i_{j-1}}(E_{i_j}) \qquad \mbox{for $j=1, \dots, N(N-1)$.}
\end{align*}
For $\cI=(i_{2N-1},\dots,i_{N(N-1)})\in \N_0^{(N-1)(N-2)}$ define 
\begin{align*}
  E_{\cI}=E_{\beta_{N(N-1)}}^{i_{N(N-1)}} \dots  E_{\beta_{2N}}^{i_{2N}}E_{\beta_{2N-1}}^{i_{2N-1}}.
\end{align*}  
By the PBW-Theorem for $\cM^+_X$ the ordered monomials $E_\cI$ for $\cI\in \N_0^{(N-1)(N-2)}$ form a $\qfield$-vector space basis of $\cM_X^+$. Moreover, the set $\{K_\beta\,|\,\beta\in Q_X\}$ is a $\qfield$-vector space basis of $U^0_X$.
Combining these facts with Proposition \ref{prop:Bc-basis} we obtain the desired PBW-Theorem for $\cB_\bc$.
\begin{thm}\label{thm:PBW-Bc}
  The set of ordered monomials
  \begin{align*}
    \{B_\cJ K_\beta E_\cI\,|\,\cJ\in \N_0^{N(N-1)}, \beta\in Q_X, \cI\in \N_0^{(N-1)(N-2)}\}
  \end{align*}
  is a $\qfield$-vector space basis of $\cB_\bc$.
\end{thm}
By Lemma \ref{lem:root-vectors} and Lemma \ref{lem:Bbetaj2} the root vectors $B_{\beta_1},\dots,B_{\beta_{N-1}}$ correspond to positive roots of $\sofrak(2N-1)$ while the root vectors $B_{\beta_N},\dots,B_{\beta_{N(N-1)}}$ correspond to negative roots. Hence it is reasonable to define
\begin{align*}
  \cB_\bc^- &= \mathrm{span}_\qfield\{B_{\beta_{N(N-1)}}^{j_{N(N-1)}}\dots B_{\beta_N}^{j_N}\}\,|\,(j_N,\dots,j_{N(N-1)})\in \N_0^{(N-1)^2}\},\\
  \cB_\bc^+&=\mathrm{span}_\qfield\{B_{\beta_{N-1}}^{j_{N-1}}\dots B_{\beta_1}^{j_1} E_{\cI}\}\,|\,(j_1,\dots,j_{N-1})\in\N_0^{N-1}, \cI\in \N_0^{(N-1)(N-2)}\}
\end{align*}
and we obtain a triangular decomposition of $\cB_\bc$.
\begin{cor}\label{cor:Bc-triang}
  The multiplication map $\cB_\bc^-\ot U^0_X \ot \cB_\bc^+\rightarrow \cB_\bc$ is a linear isomorphism. 
\end{cor}
\begin{rema}
  Note that $\cB_\bc^+$ and $\cB_\bc^-$ are not subalgebras of $\cB_\bc$. Indeed, $B_{\beta_1}=B_1$ and $B_{\beta_2}=T_2^{-1}(B_1)$ belong to $\cB_\bc^+$. However, the deformed $q$-Serre relation \eqref{eq:qSerre-Bc1} implies that $[B_{\beta_2},B_{\beta_1}]_{q^{-1}}\notin \cB_\bc^+$ because of the term $F_2\cZ_1$.
\end{rema}
\begin{rema}
  The definitions of $\cB_\bc^+$ and $\cB_\bc^-$ are symmetric as the ordered monomials $B_{\beta_{N(N-1)}}^{j_{N(N-1)}}\dots B_{\beta_{2N-1}}^{j_{2N-1}}$ form a PBW-basis of $\cM_X^-$.
\end{rema}  
\subsection{Commutation of root vectors}\label{sec:commutation}
For later use we investigate some $q$-commutators of root vectors. Let $\cM_{X,+}^+$ be the augmentation ideal of $\cM_X^+$, that is, the ideal generated by the set $\{E_i\,|\,i\in X\}$.
\begin{lem}\label{lem:BblBbk}
  For $1\le \ell<m\le N-1$ the relation $[B_{\beta_\ell},B_{\beta_m}]_q \in \cM_X \cM_{X,+}^+$ holds.
\end{lem}
\begin{proof}
  To avoid notational confusion we assume $\ell\ge 2$, but the argument for $\ell=1$ is similar. By definition of the root vectors we have
  \begin{align}
    [B_{\beta_\ell},B_{\beta_m}]_q&=[T_\ell^{-1}\dots T_2^{-1}(B_1), T_m^{-1}\dots T_2^{-1}(B_1)]_q\nonumber\\
    &=T_\ell^{-1}\dots T_2^{-1} T_m^{-1}\dots T_3^{-1}([B_1,T_2^{-1}(B_1)]_q). \label{eq:BblBbk}
  \end{align}
  By \eqref{eq:Ti-B} and \eqref{eq:qSerre-Bc1} we have
  \begin{align*}
    [B_1,T_2^{-1}(B_1)]_q\in \qfield F_2 \cZ_1 + \qfield K_2 r_2(\cZ_1) + \qfield K_2^{-1} {}_2r(\cZ_1)
  \end{align*}
  where $\cZ_1\in (\cM_X^+)_{w_X(\alpha_1)-\alpha_1}$ and $r_2(\cZ_1), {}_2r(\cZ_1)\in (\cM_X^+)_{w_X(\alpha_1)-\alpha_1-\alpha_2}$. Combining the above formula with Equation \eqref{eq:BblBbk} shows that
  \begin{align*}
    [B_{\beta_\ell},B_{\beta_m}]_q\in \cM_X (\cM_X)_\mu + \cM_X (\cM_X)_\nu
  \end{align*}
  where
  \begin{align*}
    \mu&=s_{\ell}\dots s_2 s_m\dots s_3(w_X(\alpha_1)-\alpha_1)\\
    &=w_X(s_\ell \dots s_2(\alpha_1))-s_\ell \dots s_2(\alpha_1)\\
    &= 2(\alpha_{\ell+1}+\dots+\alpha_{N-2})+\alpha_{N-1} + \alpha_{N}\\
    \\
    \nu &=s_{\ell}\dots s_2 s_m\dots s_3(w_X(\alpha_1)-\alpha_1-\alpha_2)\\
    &=\mu-s_\ell\dots s_2 s_m\dots s_3(\alpha_2)\\
    &=\mu-(\alpha_{\ell+1}+\dots + \alpha_m)
  \end{align*}
  As $0\neq \mu,\nu\in Q_X^+$, the triangular decomposition of $\cM_X$ implies that $(\cM_X)_\mu$ and $(\cM_X)_\nu$ are contained in $\cM_X \cM_{X,+}^+$. This completes the proof of the Lemma.
\end{proof}
As an immediate consequence of the above Lemma we obtain the following result.
\begin{cor}\label{lem:inBcE}
  For each $m \in \{N-1,N\}$ we have
  $[B_{\beta_{N-2}},B_{\beta_m}]_q\in \cM_X \cM_{X,+}^+$.
\end{cor}
\begin{proof}
  The case $m=N-1$ is a special case of the above Lemma. The case $m=N$ holds by symmetry between $N$ and $N-1$. 
\end{proof}
The next Lemma follows from Equation \eqref{eq:Bbetaj2} and the defining relations of $\cB_\bc$ given in Section \ref{sec:gen-rels}.
\begin{lem}\label{lem:EBBE}
  For $i\in X$ and $m=1, \dots,N-1$ we have
  \begin{align*}
    [E_i,B_{\beta_m}]= \begin{cases}
      B_{\beta_{m-1}} K_{m}^{-1} & \mbox{if $i=m$,}\\
      0 & \mbox{if $i\neq m$.}
                     \end{cases}
  \end{align*}  
\end{lem}
Similarly, we need commutation relations between the root vectors $B_{\beta_m}$ and the generators $F_i$ for $i\in X$.
\begin{lem}\label{lem:FBBF}
  The following relations hold in $\cB_\bc$:
  \begin{align}
    [B_{\beta_m},F_m]_q&=0, \qquad \qquad \mbox{for $m=1,\dots,N$}\label{eq:FBBF1}\\
    [F_{m+1},B_{\beta_m}]_{q}&=B_{\beta_{m+1}} \qquad\,\, \mbox{for $m=1,\dots,N-2$,} \label{eq:FBBF2}
  \end{align}
  Moreover, for $m=1,\dots,N$, $i\in X\setminus\{m,m+1\}$ with ($m\neq N-2$ if $i=N$) and ($m\neq N$ if $i=N-1$) the relation
  \begin{align}
    [F_i,B_{\beta_m}]&=0 \label{eq:FBBF3}
  \end{align}
  holds. 
\end{lem}
\begin{proof}
  By symmetry between $N-1$ and $N$ we may assume that $m\neq N$. To verify Equation \eqref{eq:FBBF1} we calculate
  \begin{align*}
    [B_{\beta_m},F_m]_q&=T_m^{-1}\dots T_2^{-1}([B_1,T_2\dots T_m(F_m)]_q)\\
      &=-T_m^{-1}\dots T_2^{-1}([B_1,K_2^{-1}\dots K_m^{-1}T_2\dots T_{m-1}(E_m)]_q)
  \end{align*}
  which vanishes by the defining relations of $\cB_\bc$ in Section \ref{sec:gen-rels}. Relation \eqref{eq:FBBF2} follows from Equation \eqref{eq:Bbetaj2}. For $i>m+1$ Equation \eqref{eq:FBBF3} follows again from the relations in Section \ref{sec:gen-rels}. For $i<m$ we calculate
  \begin{align*}
    [F_i,B_{\beta_m}]&=T_m^{-1}\dots T_2^{-1}([T_2\dots T_iT_{i+1}(F_i),B_{1}])\\
    &=T_m^{-1}\dots T_2^{-1}([F_{i+1},B_1])=0
  \end{align*}
  where we used the relation $T_i T_{i+1}(F_i)=F_{i+1}$ which holds by \cite[Remark 8.20]{b-Jantzen96}.
\end{proof}
By symmetry between $N-1$ and $N$, Formula \eqref{eq:FBBF2} gives $[F_N,B_{\beta_{N-2}}]_q=B_{\beta_N}$. More generally, induction over $m$ gives the following Lemma. Recall the $q$-numbers $[m]_q=(q^m-q^{-m})/(q-q^{-1})$ for $m\in \N_0$.
\begin{lem}\label{lem:BN-2F}
  For all $m\in \N_0$ the relation
  \begin{align*}
    B_{\beta_{N-2}}F_N^m = q^{-m} F_N^m B_{\beta_{N-2}} - q^{-1}[m]_q F_N^{m-1} B_{\beta_N}
  \end{align*}
  holds in $\cB_\bc$.
\end{lem}  
\section{Finite-dimensional irreducible representations of $\cB_\bc$}
  The identification of root vectors for $\cB_\bc$ and the triangular decomposition in Corollary \ref{cor:Bc-triang} allow us to follow a Verma module approach towards the classification of finite dimensional $\cB_\bc$-modules. To this end we show in Section \ref{sec:missing} that $\{B_{\beta_{N-1}}, B_{\beta_{N}}, (K_NK_{N-1}^{-1})^{\pm 1}\}$ act as a copy of $U_q(\slfrak(2))$ on suitable subspaces of representations. We consider embeddings of Verma modules in Section \ref{sec:submodules} and complete the classification of finite dimensional, simple $\cB_\bc$ modules in Section \ref{sec:filt-grad} using a filtered-graded argument. First, however, we briefly discuss signs for the action of $U^0_X$ which can occur in complete analogy to the quantum group case in \cite[5.1]{b-Jantzen96}.
\subsection{Representations of type $\one$}
Let $V$ be a finite-dimensional representation of the $\qfield$-algebra $\cB_\bc$. By restriction, $V$ can be considered as a module over the subalgebra $\cM_X$. Hence, the action of $U_X^0$ on $V$ is diagonalisable.

Assume that $\mathrm{char}(\qfield)\neq 2$. Let $P_{2N-1}=\Hom(Q^\vee_{2N-1},\Z)$ be the weight lattice of $\sofrak(2N-1)$. For every group homomorphism $\chi\in \Hom(Q^\vee_{2N-1},\{\pm 1\})$ and any $\lambda\in P_{2N-1}$ define
\begin{align}\label{eq:V-lambda-chi}
  V_{\lambda,\chi}=\{v\in V\,|\, K_{\eta(h)} v = \chi(h) q^{\lambda(h)}v \quad \mbox{for all $h\in Q^\vee_{2N-1}$}\}.
\end{align}\label{eq:V-decomp}
Following \cite[5.2]{b-Jantzen96} we have
\begin{align}
  V=\bigoplus_{\chi} V^\chi \qquad \mbox{where}\quad V^\chi=\bigoplus_{\lambda\in P_{2N-1}} V_{\lambda,\chi}
\end{align}
where the first sum is over all group homomorphisms $\chi\in \Hom(Q^\vee_{2N-1},\{\pm 1\})$. As the generator $B_1$ of $\cB_\bc$ $q$-commutes with all $K_i$ for $i\in X$, the $\cM_X$-submodule $V^\chi$ is a $\cB_\bc$-submodule of $V$. We say that $V$ is of type $\chi$ if $V=V^\chi$. 
Let $\cB_\bc\mbox{-mod}_{f,\chi}$ denote the category of finite dimensional $\cB_\bc$ modules of type $\chi$.

For every $\chi\in \Hom(Q^\vee_{2N-1},\{\pm 1\})$ there exists a one-dimensional $\Uq$-module $\qfield_\chi=\qfield {\mathbbm 1}_\chi$ defined by
\begin{align*}
  u {\mathbbm 1}_\chi &= \vep(u) {\mathbbm 1}_\chi \quad\mbox{for $u=E_i, F_i$ for $i\in I$, and for $u=K_1$,}\\
  K_{\eta(h)}{\mathbbm 1}_\chi &= \chi(h) {\mathbbm 1}_\chi \quad \mbox{for all $h\in Q^\vee_{2N-1}$.}
\end{align*}  
As $\cB_\bc\subset \Uq$ is a right coideal subalgebra, the category of finite dimensional $\cB_\bc$-modules is a right module category over the category of finite dimensional $\Uq$-modules. As in \cite[5.2]{b-Jantzen96} this allows us to formulate an equivalence of categories. Let ${\mathbf 1}\in \Hom(Q^\vee_{2N-1},\{\pm 1\})$ be the trivial group homomorphism defined by ${\mathbf 1}(h)=1$ for all $h\in Q^\vee_{2N-1}$.
\begin{prop}\label{prop:cat-equiv}
  Assume that $\mathrm{char}(\qfield)\neq 2$ and let $\chi\in \Hom(Q^\vee_{2N-1},\{\pm 1\})$. The functor
  \begin{align*}
    \cB_\bc\mbox{\upshape{-mod}}_{f,{\mathbf 1}} \rightarrow \cB_\bc\mbox{\upshape{-mod}}_{f,\chi}, \qquad V\mapsto V\ot \qfield_\chi
  \end{align*}
  is an equivalence of categories. 
\end{prop}  
Recall the involutive algebra anti-automorphism $\sigma^B$ defined by \eqref{eq:sigmaB-def}. Let $V$ be a finite dimensional $\cB_\bc$-module. We define a $\cB_\bc$-module structure on the linear dual space $V^\ast$ by
\begin{align}\label{eq:dual-def}
  (bf)(v)=f(\sigma^B(b)v) \qquad \mbox{for all $f\in V^\ast$, $b\in \cB_\bc$, $v\in V$.}
\end{align}  
Note that for any $\chi\in \Hom(Q^\vee_{2N-1},\{\pm 1\})$ the category $\cB_\bc\mbox{-mod}_{f,\chi}$ is closed under taking duals. 

Proposition \ref{prop:cat-equiv} shows that we lose no information if we restrict to type $\one$ $\cB_\bc$-modules. Moreover, this category is closed under taking duals. We shall henceforth only consider finite dimensional $\cB_\bc$-modules of type $\one$.

\begin{rema}
  If $\mathrm{char}(\qfield)=2$ then all types coincide. In this case the category of finite dimensional $\cB_\bc$-modules of type $\one$ is just the category of all finite dimensional $\cB_\bc$-modules.
\end{rema}  
\subsection{Existence of a highest weight vector}
Let $V$ be a finite dimensional $\cB_\bc$-module of type $\one$. For any $\lambda\in P_{2N-1}$ we write $V_\lambda=V_{\lambda,\one}$ and hence the second equation in \eqref{eq:V-decomp} becomes
$$V=\bigoplus_{\lambda\in P_{2N-1}}V_\lambda.$$
We call the elements of $V_\lambda$ weight vectors of weight $\lambda$.
The commutation relations in Lemma \ref{lem:root-vectors} imply the following result.
\begin{lem}
  Let $V$ be a representation of $\cB_\bc$. For $i=2,\dots,N-1$ and $\lambda\in P_{2N-1}$ the following relations hold:
  \begin{align*}
    E_i V_\lambda &\subseteq V_{\lambda+\gamma_{i-1}}, &
    F_i V_\lambda &\subseteq V_{\lambda-\gamma_{i-1}},\\
    B_{\beta_{N-1}} V_\lambda & \subseteq V_{\lambda+ \gamma_{N-1}}, &
    B_{\beta_{N}} V_\lambda & \subseteq V_{\lambda- \gamma_{N-1}},\\
    E_N V_\lambda &\subseteq V_{\lambda+\gamma_{N-2}+2\gamma_{N-1}}, &
    F_N V_\lambda &\subseteq V_{\lambda-\gamma_{N-2}-2\gamma_{N-1}}.
  \end{align*}
\end{lem}  
The above Lemma implies the existence of a highest weight vector.
\begin{lem}\label{lem:hwv}
  Let $V\neq \{0\}$ be a finite dimensional $\cB_\bc$-module. There exists a weight vector $v\in V\setminus\{0\}$ such that
  \begin{align*}
    E_iv&=0 \qquad \mbox{for $i=2, \dots, N$},\\
    B_{\beta_{N-1}}v&=0.
  \end{align*}
  Moreover, the weight $\lambda\in P_{2N-1}$ of $v$ satisfies $\lambda(h_{\gamma_j})\ge 0$ for $j=1, \dots, N-2$.
\end{lem}
\begin{proof}
  The second statement follows from the fact that $K_{\eta(h_{\gamma_j})}=K_{j+1}$ for $j=1,\dots,N-2$ and that $\{E_{j+1}, F_{j+1}, K^{\pm 1}_{j+1}\}$ generates a copy of $U_q(\slfrak(2))$.
\end{proof}  
\subsection{The commutator $[B_{\beta_{N-1}},B_{\beta_N}]$}
Lemma \ref{lem:hwv} shows that the weight $\lambda$ of a highest weight vector $v$ in a finite-dimensional representation $V$ of $\cB_\bc$ is dominant with respect to $\{h_{\gamma_1}, \dots, h_{\gamma_{N-2}}\}$. To show that $\lambda$ is also dominant with respect to $h_{\gamma_{N-1}}$, we need to investigate the commutator
  \begin{align}\label{eq:Om-def}
     \Omega:=[B_{\beta_{N-1}},B_{\beta_N}].
  \end{align}
  To simplify notation, we define $\Omega_1=[F_1,E_1]$ and
  \begin{align}
    \Omega_j&=[T_j^{-1}\dots T_2^{-1}(F_1),T_j\dots T_2(E_1)] \qquad \mbox{for $j=2, \dots, N-1$,}\label{eq:Omj-def}\\
    \Omega_N&=[T_N^{-1}T_{N-2}^{-1}\dots T_2^{-1}(F_1),T_NT_{N-2}\dots T_2(E_1)].\label{eq:OmN-def}
  \end{align}
The commutator $\Omega$ can be expressed in terms of $\Omega_{N-1}$ and $\Omega_N$.  
\begin{lem}\label{lem:Om1}
  The commutator $\Omega$ given by \eqref{eq:Om-def} satisfies the relation
  \begin{align*}
    \Omega= c_1\big(\Omega_N K_{N-1}^{-1}-\Omega_{N-1} K_N^{-1} \big) K_1^{-1}\dots K_{N-2}^{-1}.
  \end{align*}
  In particular, we have $\Omega\in \Uq_0$.
\end{lem}  
\begin{proof}
  Set $A=T_{N-2}^{-1}\dots T_2^{-1}(F_1)$, $C=T_{N-2} \dots T_2(E_1)$ and $K=K_1K_2 \dots K_{N-2}$. By \eqref{eq:ij-ji} for $j\in \{N-1,N\}$ we have
  \begin{align*}
    T_{j}^{-1}(A)&=T_1 T_2\dots T_{N-2}(F_j), & T_j(C)&= T_1^{-1}T_2^{-1}\dots T_{N-2}^{-1}(E_j)
  \end{align*}
  and hence
  \begin{align*}
     [T_N^{-1}(A),T_{N-1}^{-1}(A)]&=0, & [T_N^{-1}(C),T_{N-1}^{-1}(C)]&=0. 
  \end{align*}
  Moreover, $[T_{N-1}^{-1}(A),K K_N]=0=[T_{N}(C),K K_{N-1}]$ for weight reasons.
  With these preparations we calculate
  \begin{align*}
    \Omega&=[B_{\beta_{N-1}}, B_{\beta_N}]\\
      &=[T_{N-1}^{-1}(A) - c_1 T_N(C)K^{-1}K_{N-1}^{-1}, T_{N}^{-1}(A) - c_1 T_{N-1}(C)K^{-1}K_{N}^{-1}]\\
    &=-c_1[T_{N-1}^{-1}(A),T_{N-1}(C)]K^{-1}K_{N}^{-1} - c_1 [T_N(C), T_N(A)]K^{-1}K_{N-1}^{-1}\\
    &= - c_1 \Omega_{N-1}K_N^{-1}K^{-1} + c_1 \Omega_N K_{N-1}^{-1}K^{-1}
  \end{align*}
  which proves the desired formula.
\end{proof}
By Equations \eqref{eq:Omj-def}, \eqref{eq:OmN-def} and Lemma \ref{lem:Om1} we have $\Omega\in \Uq_0\cap \cB_\bc$. By Proposition \ref{prop:Bc-basis} or \cite[Proposition 6.2]{a-Kolb14} we have $\Uq_0\cap \cB_\bc=\cM_{X,0}$. Hence, $\Omega$ is an element of weight $0$ in $\cM_X$, that is,
\begin{align}\label{eq:OmMX}
  \Omega\in \cM_{X,0}.
\end{align}  
The triangular decomposition \eqref{eq:triang} of $\Uq$ gives a direct sum decomposition
\begin{align*}
  U_0= U^0 \oplus \bigoplus_{\nu>0} U^-_{-\nu} U^0 U^+_\nu.
\end{align*}
Let $\pi: \Uq \rightarrow U^0$ be the projection with respect to this decomposition, see \cite[6.2]{b-Jantzen96}. For our purposes, it suffices to know $\pi(\Omega)$. As a first step we determine $\pi(\Omega_j)$ for the elements $\Omega_j$ given by \eqref{eq:Omj-def}, \eqref{eq:OmN-def}.
For any $j\in I$ there exists a well-defined algebra automorphism $\Psi_j:U^0\rightarrow U^0$ such that
\begin{align*}
  \Psi_j(K_i) = q^{\delta_{i,j}} K_i.
\end{align*}
\begin{lem}\label{lem:Om2}
  For all $j\in \{1,\dots, N-2\}$ we have the iterative formula
  \begin{align*}
    \pi(\Omega_{j+1}) =\pi(\Omega_j)\frac{K_{j+1} -(2-q^{-2})K_{j+1}^{-1}}{q-q^{-1}} - q \Psi_j(\pi(\Omega_j))\frac{K_{j+1}-K_{j+1}^{-1}}{q-q^{-1}}.
  \end{align*}
\end{lem}
\begin{proof}
  For $j\in \{1,\dots,N-2\}$ define $A_j=T_j^{-1}\dots T_2^{-1}(F_1)$ and $C_j=T_j\dots T_2(E_1)$, and in particular $A_1=F_1$ and $C_1=E_1$. We have $\Omega_j=[A_j,C_j]$ and $\Omega_{j+1}=[[F_{j+1},A_j]_q,[E_{j+1},C_j]_{q^{-1}}]$. Using the relations $[F_{j+1},C_j]=0=[E_{j+1},A_j]$ we obtain
  \begin{align*}
    \pi&(\Omega_{j+1})=-\pi\big([E_{j+1},C_j]_{q^{-1}}\,[F_{j+1},A_j]_q\big)\\
    &= - \pi\big(E_{j+1} F_{j+1} C_j A_j {-} q^{-1} C_jE_{j+1} F_{j+1} A_j {-} q E_{j+1} C_j A_j F_{j+1} {+} C_j A_j E_{j+1} F_{j+1}\big)\\
    &=-\frac{K_{j+1}-K_{j+1}^{-1}}{q-q^{-1}} \pi([C_j, A_j]) + q^{-1}\pi([C_j, A_j]) \frac{q K_{j+1}-q^{-1}K_{j+1}^{-1}}{q-q^{-1}}\\
    &\qquad  + q \pi\big(E_{j+1} \pi([C_j,A_j])F_{j+1}\big) - \pi([C_j,A_j]) \frac{K_{j+1}-K_{j+1}^{-1}}{q-q^{-1}}\\
    &= 2 \pi(\Omega_j)\frac{K_{j+1}-K_{j+1}^{-1}}{q-q^{-1}} - \pi(\Omega_j) \frac{K_{j+1}-q^{-2}K_{j+1}^{-1}}{q-q^{-1}} - q \Psi_j(\pi(\Omega_j)) \frac{K_{j+1}-K_{j+1}^{-1}}{q-q^{-1}}
  \end{align*}
  which gives the desired formula.
\end{proof}
The above lemma allows us to compare $\pi(\Omega_j)$ and $\Psi_j(\pi(\Omega_j))$.
\begin{lem}\label{lem:Om3}
  For all $j\in \{1,\dots, N-1\}$ we have
  \begin{align}\label{eq:OmPsiOm}
    \pi(\Omega_j)-q \Psi_j(\pi(\Omega_j))=(-1)^{j-1}q^j K_1 \dots K_j.
  \end{align}  
\end{lem}
\begin{proof}
  We perform induction on $j$. For $j=1$ we have $\Omega_1=\frac{K_1^{-1}-K_1}{q-q^{-1}}$ and Equation \eqref{eq:OmPsiOm} holds. For $j=1, \dots, N-2$ we now use Lemma \ref{lem:Om2} and the fact that $\pi(\Omega_j)$ is invariant under $\Psi_{j+1}$ to calculate
  \begin{align*}
    \pi(\Omega_{j+1})-&q \Psi_{j+1}(\pi(\Omega_{j+1}))\\
    =&\pi(\Omega_j)\left[\frac{K_{j+1}-(2-q^{-1})K_{j+1}^{-1}}{q-q^{-1}} - q\frac{qK_{j+1}-(2-q^{-1})q^{-1}K_{j+1}^{-1}}{q-q^{-1}} \right]\\
    & - q\Psi_j(\pi(\Omega_j))\left[\frac{K_{j+1}-K_{j+1}^{-1}}{q-q^{-1}} - q\frac{qK_{j+1}-q^{-1}K_{j+1}^{-1}}{q-q^{-1}} \right]\\
    =&\big(\pi(\Omega_j) - q \Psi_j(\pi(\Omega_j))\big) (-q)K_{j+1}.
  \end{align*}
  This formula provides the induction step.
\end{proof}  
Using the three lemmas above, we obtain the desired expression for $\pi(\Omega)$.
\begin{prop}\label{prop:piOmega}
  The commutator $\Omega$ given by \eqref{eq:Om-def} satisfies the relation
  \begin{align*}
    \pi(\Omega)=c_1(-1)^{N-1}q^{N-2}\frac{K_NK_{N-1}^{-1}- K_{N-1}K_N^{-1}}{q-q^{-1}}.
  \end{align*}  
\end{prop}  
\begin{proof}
  Using Lemmas \ref{lem:Om1} and \ref{lem:Om2} we get
  \begin{align*}
    \pi(\Omega)=& c_1\big(\pi(\Omega_N)K_{N-1}^{-1} - \pi(\Omega_{N-1})K_N^{-1} \big) K_1^{-1} \dots K_{N-2}^{-1}\\
    =&c_1 \Big[\pi(\Omega_{N-2})\frac{K_N{-}(2{-}q^{-2})K_N^{-1}}{q-q^{-1}}\\&\qquad\qquad - q \Psi_{N-2}(\pi(\Omega_{N-2}))\frac{K_N{-}K_N^{-1}}{q-q^{-1}}\Big]K_1^{-1}\dots K_{N-2}^{-1} K_{N-1}^{-1}\\
    &-c_1\Big[\pi(\Omega_{N-2})\frac{K_{N-1}{-}(2{-}q^{-2})K_{N-1}^{-1}}{q-q^{-1}}\\&\qquad\qquad - q \Psi_{N-2}(\pi(\Omega_{N-2}))\frac{K_{N-1}{-}K_{N-1}^{-1}}{q-q^{-1}}\Big]K_1^{-1}\dots K_{N-2}^{-1} K_N^{-1}\\
    =&c_1 \big(\pi(\Omega_{N-2}) - q \Psi_{N-2}(\pi(\Omega_{N-2}))\big)\frac{K_N K_{N-1}^{-1} - K_{N-1}K_N^{-1}}{q-q^{-1}} K_1^{-1}\dots K_{N-2}^{-1}.
  \end{align*}
Now the statement of the proposition follows from Lemma \ref{lem:Om3}.
\end{proof}  
\subsection{The missing $\slfrak_2$-triple}\label{sec:missing}
Let $V$ be a representation of $\cB_\bc$. Proposition \ref{prop:piOmega} suggests that the triple $B_{\beta_{N-1}}, B_{\beta_N}, K_N K_{N-1}^{-1}$ provides a representation of $U_q(\slfrak_2)$ on a suitable subspace of $V$. We define
\begin{align}\label{eq:HV-def}
  H(V)=\{v\in V\,|\,B_{\beta_{N-2}}v=E_iv=0 \quad \mbox{for all $i\in X$}\}.
\end{align}
Recall from Lemma \ref{lem:EBBE} that $B_{\beta_{N-2}}=[E_{N-1},B_{\beta_{N-1}}]K_{N-1}$. Hence, Lemma \ref{lem:hwv} implies that $H(V)\neq \{0\}$  if $0<\dim(V)<\infty$.
\begin{lem}\label{lem:HV}
  Let $V$ be a $\cB_\bc$-module. The subspace $H(V)$ is invariant under the action of the root vectors $B_{\beta_{N-1}}$ and $B_{\beta_N}$.
\end{lem}  
\begin{proof}
  Fix $v\in H(V)$. Recall the root vectors $B_{\beta_j}$ defined by \eqref{eq:Bbetaj}. For $\ell\in \{N-1,N\}$ it follows from \eqref{eq:Bbetaj2} that
  \begin{align}\label{eq:Bbetaell}
     B_{\beta_\ell}=[F_\ell,B_{\beta_{N-2}}]_q.
  \end{align}  
  Hence we have $B_{\beta_\ell}v=-q B_{\beta_{N-2}}F_\ell v$ by definition of $H(V)$. To prove the Lemma, it remains to verify the relations
  \begin{align}
    E_i B_{\beta_{N-2}} F_\ell v &= 0 \qquad \mbox{for all $i\in X$,}\label{eq:goal1}\\
    B^2_{\beta_{N-2}} F_\ell v &=0. \label{eq:goal2}
  \end{align}
  Again by \eqref{eq:Bbetaj2} we have
  \begin{align*}
    [E_i,B_{\beta_{N-2}}]=\begin{cases}
              qK_{N-2}^{-1}B_{\beta_{N-3}}& \mbox{if $i=N-2$,}\\
              0 & \mbox{if $i\in X\setminus\{N-2\}$.}  
                      \end{cases}
  \end{align*}
  The above relation implies Equation \eqref{eq:goal1}. Indeed, for $i=N-2$ we obtain
  \begin{align*}
     E_{N-2}B_{\beta_{N-2}}F_\ell v = B_{\beta_{N-2}}E_{N-2}F_\ell v + qK_{N-2}^{-1}B_{\beta_{N-3}}F_\ell v = 0
  \end{align*}
  as $[E_{N-2},F_\ell]=[B_{\beta_{N-3}},F_\ell]=0$. And for $i=\ell$ we have
  \begin{align*}
    E_\ell B_{\beta_{N-2}}F_\ell v = B_{\beta_{N-2}}\frac{K_\ell-K_\ell^{-1}}{q-q^{-1}}v= \frac{qK_\ell-q^{-1}K_\ell^{-1}}{q-q^{-1}}B_{\beta_{N-2}}v=0.
  \end{align*}
  To verify Equation \eqref{eq:goal2}, observe that Equation \eqref{eq:Bbetaell} and the definition of $H(V)$ imply that
  \begin{align*}
    B^2_{\beta_{N-2}} F_\ell v = -q^{-1} B_{\beta_{N-2}} B_{\beta_\ell}v=-q^{-1}[B_{\beta_{N-2}},B_{\beta_\ell}]_q v.
  \end{align*}
  This expression vanishes by Corollary \ref{lem:inBcE} and by the definition of $H(V)$.
\end{proof}
Lemma \ref{lem:HV} and Proposition \ref{prop:piOmega} provide us with the desired representation of $U_q(\slfrak_2)$ on the subspace $H(V)$. Observe that the subspace $H(V)$ is also invariant under $U^0_X$. 
\begin{prop}\label{prop:Uqsl2}
  Let $V$ be a $\cB_\bc$-module and $a_{N-1}, a_N\in \qfield$ satisfying the relation $a_{N-1}a_N =(-1)^{N-1} c_1^{-1}q^{2-N}$. There exists an algebra homomorphism $U_q(\slfrak_2)\rightarrow \mathrm{End}(H(V))$ such that
  \begin{align*}
     E\mapsto a_{N-1}B_{\beta_{N-1}}, \quad F\mapsto a_NB_{\beta_N}, \quad K \mapsto K_N K_{N-1}^{-1}.
  \end{align*}  
\end{prop}  
\begin{proof}
  By Lemma \ref{lem:HV} the subspace $H(V)$ is invariant under $B_{\beta_{N-1}}$ and $B_{\beta_N}$. By \eqref{eq:OmMX} the commutator $\Omega$ defined by \eqref{eq:Om-def} satisfies
  \begin{align*}
    \Omega v = \pi(\Omega) v \qquad \mbox{for all $v\in H(V)$.}
  \end{align*}
  Now the statement of the Proposition follows from Proposition \ref{prop:piOmega}.
\end{proof}  
Proposition \ref{prop:Uqsl2} implies that the weight $\lambda\in P_{2N-1}$ in Lemma \ref{lem:hwv} is dominant.
\begin{cor}\label{cor:gammaj-dom}
   Let $V$ be a finite-dimensional representation of $\cB_\bc$ and $v\in V$ a weight vector of weight $\lambda\in P_{2N-1}$ such that $B_{\beta_{N-1}}v=E_iv=0$ for all $i\in X$. Then $\lambda(h_{\gamma_j})\in \N_0$ for all $j=1, \dots, N-1$.
\end{cor}  
\subsection{Verma modules for $\cB_\bc$}\label{sec:Verma}
For any $\lambda\in P_{2N-1}$ consider the left ideal
\begin{align*}
  J_{\lambda}=\sum_{i=2}^N \bigg(\cB_\bc E_i +\cB_\bc(K_i-q^{\lambda(\eta^{-1}(h_{\alpha_i}))})\bigg) + \cB_\bc B_{\beta_{N-1}} \subseteq \cB_\bc
\end{align*}
and define a corresponding Verma module $M(\lambda)$ for $\cB_\bc$ by
\begin{align*}
  M(\lambda) = \cB_\bc/J_\lambda.
\end{align*}
We write $v_\lambda\in M(\lambda)$ to denote the coset $1+J_\lambda$.
\begin{rema}
  By Lemma \eqref{lem:root-vectors} the root vectors in $\cB_\bc$ corresponding to positive, simple roots are given by $E_2,\dots, E_{N-1}, B_{\beta_{N-1}}$. Hence, it would be natural to consider the left ideal
  \begin{align*}
    I_\lambda=\sum_{i=2}^{N-1} \bigg(\cB_\bc E_i +\cB_\bc(K_i{-}q^{\lambda(h_{\gamma_{i-1}})})\bigg) + \cB_\bc B_{\beta_{N-1}} + \cB_\bc(K_{N}K_{N-1}^{-1}{-}q^{\lambda(h_{\gamma_{N-1}})}).
  \end{align*}
 It is not obvious from the relations of $\cB_\bc$ that $E_N\in I_\lambda$ and hence it is convenient to work with $J_\lambda$ instead. At $q=1$ we have $I_\lambda=J_\lambda$.
\end{rema}  
The PBW-Theorem for $\cB_\bc$ allows us to prove a PBW-Theorem for the Verma module $M(\lambda)$. Recall the definition $\eqref{eq:BJ-def}$ of the PBW-monomial $B_\cJ$ for  $\cJ=(j_1,\dots,j_{N(N-1)})\in \N_0^{N(N-1)}$. By Lemmma \ref{lem:EBBE} the root vectors $B_{\beta_j}$ satisfy the relation
\begin{align*}
  B_{\beta_{j-1}}K_{j}^{-1}=[E_{j},B_{\beta_j}] \qquad \mbox{for $j=2,\dots,N-1$.}
\end{align*}  
Hence $B_\cJ v_\lambda=0$ unless $j_1=\cdots=j_{N-1}=0$. By the following theorem, the remaining vectors $B_\cJ v_\lambda$ form a basis of the Verma module $M(\lambda)$.
\begin{thm}\label{thm:PBWMl}
  For any $\lambda \in P_{2N-1}$ the family of vectors
  \begin{align*}
     \{B_\cJ v_\lambda\,|\, \cJ\in \N_0^{N(N-1)}, j_1=\dots=j_{N-1}=0\}
  \end{align*}
  is a basis of the Verma module $M(\lambda)$.
\end{thm}
\begin{proof}
  Let $U_X^0(\lambda)$ denote the maximal ideal of $U_X^0$ generated by the elements $K_i-q^{\lambda(\eta^{-1}(h_{\alpha_i}))}$ for $i\in X$. Define linear subspaces of $\cB_\bc$ by
  \begin{align*}
    M&=\mathrm{span}_\qfield\{B_\cJ\,|\,\cJ\in\N_0^{N(N-1)}, j_1=\dots=j_{N-1}=0\},\\
    M'&=\mathrm{span}_\qfield\{B_\cJ\,|\,\cJ\in\N_0^{N(N-1)}, \exists i\in\{1,\dots,N-1\} \mbox{ such that } j_i\neq 0\},\\
    N&=\mathrm{span}_\qfield\{B_\cJ\,|\,\cJ\in\N_0^{N(N-1)}\}U_X^0(\lambda),\\
    P&=\cB_\bc \cM_{X,+}^+
  \end{align*}
  where $\cM_{X,+}^+$ denotes the augmentation ideal of $\cM_X^+$ defined at the beginning of Section \ref{sec:commutation}. By Theorem \ref{thm:PBW-Bc} we have a direct sum decomposition
  \begin{align*}
    \cB_\bc = M \oplus M' \oplus N \oplus P.
  \end{align*}
  By definition of $J_\lambda$ we have $N, P \subset J_\lambda$. By Lemma \ref{lem:EBBE} we know that $B_{\beta_k}\in J_\lambda$ for $k=1, \dots, N-1$. Hence $M'\subset J_\lambda$ and thus $M'\oplus N \oplus P \subset J_\lambda$.

  It remains to verify the opposite inclusion. To this end, it suffices to show that $bB_{\beta_{N-1}}\in M'\oplus N\oplus P$ for all $b\in \cB_\bc$. More generally, we will show by induction on $j$ that
  \begin{align}\label{induction-goal}
     b B_{\beta_j}\in M'\oplus N \oplus P \quad \mbox{for $j=1, \dots, N-1$}
  \end{align}
  for any $b\in \cB_\bc$.  
To prove \eqref{induction-goal} we may assume without loss of generality that $b=B_\cJ K_\beta E_\cI$ for some $\cJ\in \N_0^{N(N-1)}$, $\beta\in Q_X$ and $\cI\in \N_0^{(N-1)(N-2)}$, see Theorem \ref{thm:PBW-Bc}. We fist consider the case $j=1$. In this case we may assume $\cI=(0,\dots,0)$, because otherwise $b B_{\beta_1}=bB_1\in P$. Moreover, we may assume $K_\beta=1$ because $K_\beta$ and $B_1$ always $q$-commute. As $B_\cJ B_{\beta_1}\in M'$ by definition of the PBW-basis, we have established \eqref{induction-goal} for $j=1$.

Now assume that \eqref{induction-goal} holds for $j=1, \dots,\ell-1<N-1$ and consider $bB_{\beta_\ell}$. By Lemma \ref{lem:EBBE} we may again assume that $\cI=(0,\dots,0)$, and again we may $q$-commute the factor $K_\beta$ to the right hand side of $B_{\beta_\ell}$ and replace it by a power of $q$. Hence we may assume that $b=B_\cJ$ for some $\cJ\in \N^{N(N-1)}$. The relation $B_{\cJ}B_{\beta_\ell}\in M'\oplus N\oplus P$ now follows from Lemma \ref{lem:BblBbk} and the induction hypothesis. This completes the proof of the theorem.
\end{proof}
\subsection{Proper submodules of $M(\lambda)$}\label{sec:submodules}
By Lemma \ref{lem:hwv} and Corollary \ref{cor:gammaj-dom}, every finite-dimensional, simple $\cB_\bc$-module (of type ${\mathbf 1}$) is given as a quotient of $M(\lambda)$ by a maximal proper submodule $N(\lambda)\subset M(\lambda)$ for some dominant, integral weight $\lambda\in P_{2N-1}^+$. For the classification of finite-dimensional, simple $\cB_\bc$-modules, it remains to show that $\dim(M(\lambda)/N(\lambda))<\infty$ for all $\lambda\in P_{2N-1}^+$. To gain some understanding of the submodule $N(\lambda)$ we first construct homomorphism of Verma modules, in analogy to \cite[Lemma 5.6]{b-Jantzen96}. Recall that $\eta(h_{\gamma_i})=h_{\alpha_{i+1}}$ for $1\le i \le N-2$ and that hence, in this case, $F_{i+1}\in \cB_\bc$ is a negative root vector corresponding to the negative simple root $-\gamma_i$. Moreover, from \eqref{eq:Bbetaj} or \eqref{eq:Bbetaj2}, recall the definition of the root vector $B_{\beta_N}$ corresponding to the negative simple root $-\gamma_{N-1}$.
\begin{prop}\label{prop:Verma-hom}
  Let $\lambda\in P_{2N-1}$ and $i\in \{1,2, \dots, N-1\}$ with $n_i:=\lambda(h_{\gamma_i})\ge 0$. There is a homomorphism of $\cB_\bc$-modules
  $\varphi_i:M(\lambda-(n_i+1)\gamma_i)\rightarrow M(\lambda)$ with
  \begin{align*}
    \varphi_i(v_{\lambda-(n_i+1)\gamma_i})=\begin{cases}
                            F_{i+1}^{n_i+1}v_\lambda & \mbox{if $i\le N-2$,}\\
                            B_{\beta_{N}}^{n_{N-1}+1}v_\lambda & \mbox{if $i=N-1$.}
                                     \end{cases}
  \end{align*}  
\end{prop}  
\begin{proof}
  First consider the case $i\le N-2$. By the universal property of $M(\lambda-(n_i+1)\gamma_i)$ we need to verify the following two relations
  \begin{align}
    E_jF_{i+1}^{n_i+1}v_\lambda&=0 \qquad \mbox{for $j\in X$,}\label{eq:hw1}\\
    B_{\beta_{N-1}}F_{i+1}^{n_i+1}v_\lambda&=0.\label{eq:hw2}
  \end{align}
  The proof of Equation \eqref{eq:hw1} is identical to the proof of \cite[Lemma 5.6]{b-Jantzen96}. Equation \eqref{eq:hw2} follows from the commutation relations $[B_{\beta_{N-1}},F_{N-1}]_q=0$ and $[B_{\beta_{N-1}},F_i]=0$ which hold by Lemma \ref{lem:FBBF}.

  It remains to consider the case $i=N-1$. In this case, again by the universal property of $M(\lambda-(n_{N-1}+1)\gamma_{N-1})$, we need to verify the following two relations
  \begin{align}
    E_jB_{\beta_N}^{n_{N-1}+1}v_\lambda&=0 \qquad \mbox{for $j\in X$,}\label{eq:hw3}\\
    B_{\beta_{N-1}}B_{\beta_N}^{n_{N-1}+1}v_\lambda&=0.\label{eq:hw4}
  \end{align}
  By symmetry between $N-1$ and $N$, Lemma \ref{lem:EBBE} implies that
  \begin{align}\label{eq:EBBEN}
    [E_j,B_{\beta_N}]=\begin{cases}
                       B_{\beta_{N-2}}K_{N}^{-1} & \mbox{if $j=N$},\\
                       0 & \mbox{if $j\neq N$.}  
                    \end{cases}
  \end{align}
  This implies Equation \eqref{eq:hw3} in the case $j\neq N$. To address the case $j=N$ we prove by induction that $E_N B_{\beta_N}^av_\lambda =0$ for all $a\in \N_0$. In the case $a=0$ this holds by definition of $M(\lambda)$. Assuming now that $E_N B_{\beta_N}^dv_\lambda=0$ for all $d\le a$ we calculate
  \begin{align*}
    E_NB_{\beta_{N}}^{a+1}v_\lambda=[E_N,B_{\beta_N}]B_{\beta_N}^av_\lambda 
    \stackrel{\eqref{eq:EBBEN}}{=}qK_N^{-1}B_{\beta_{N-2}}B_{\beta_N}^av_\lambda.
  \end{align*}
  This expression vanishes by Corollary \ref{lem:inBcE} and induction hypothesis, and hence Equation \eqref{eq:hw3} also holds for $j=N$.

  To verify Equation \eqref{eq:hw4}, we can calculate inside the subspace $H(M(\lambda))$ defined by \eqref{eq:HV-def}. By Proposition \eqref{prop:Uqsl2}, the root vectors $B_{\beta_{N-1}}$ and $B_{\beta_N}$ give rise to a representation of $U_q(\slfrak_2)$ on $H(M(\lambda))$ up to scaling. Hence Equation \eqref{eq:hw4} can again be verified by the calculation in the proof of \cite[Lemma 5.6]{b-Jantzen96}
\end{proof}
By the above proposition, the submodule $\widetilde{N}(\lambda):=\sum_{i=1}^{N-1}\im(\varphi_i)$ of $M(\lambda)$ is proper. Ideally, we would like to adapt the arguments in \cite[5.7-5.9]{b-Jantzen96} to show that $\widetilde{L}(\lambda)=M(\lambda)/\widetilde{N}(\lambda)$ is finite-dimensional. In attempting to do so, however, we encounter some problems. In particular, the triple $B_{\beta_{N-1}}, B_{\beta_N}, K_N K_{N-1}^{-1}$ provides a copy of $U_q(\slfrak_2)$ only on the subspace $H(\widetilde{L}(\lambda))$ and not on all of $\widetilde{L}(\lambda)$. Hence it is unclear how to adapt the arguments in  \cite[5.7-5.9]{b-Jantzen96} to show that the set of weights of $\widetilde{L}(\lambda)$ is stable under the Weyl group.

For this reason we take a different approach. We will consider a quotient of $M(\lambda)$ by a submodule which at first sight seems larger than $\widetilde{N}(\lambda)$. The following proposition will provide the main step in our argument.
\begin{prop}\label{prop:FNvl}
  Let $\lambda\in P_{2N-1}$ with $n_{N-1}:=\lambda(h_{\gamma_{N-1}})\in \N_0$ and $n_{N-2}:=\lambda(h_{\gamma_{N-2}})\in \N_0$ and set $n_N=n_{N-1}+n_{N-2}$. The submodule $\cB_\bc F_N^{n_N+1}v_\lambda \subset M(\lambda)$ is contained in $\sum_{\mu<\lambda} M(\lambda)_\mu$.
\end{prop}  
We will prove Proposition \ref{prop:FNvl} in Section \ref{sec:FNvl}. By the above proposition the submodule 
\begin{align*}
  \widetilde{N}(\lambda) + \cB_\bc F_N^{n_N+1}v_\lambda
\end{align*}
of $M(\lambda)$ is contained in $\sum_{\mu<\lambda} M(\lambda)_\mu$ and hence proper. We will show in Section \ref{sec:filt-grad} using filtered-graded arguments that the quotient $M(\lambda)/ \big(\widetilde{N}(\lambda) + \cB_\bc F_N^{n_N+1}v_\lambda\big)$ is finite-dimensional. 
\begin{rema}
  The element $F_N^{m+1}v_\lambda\in M(\lambda)$ is not annihilated by the positive simple root vector $B_{\beta_{N-1}}$ for any $m\in \N_0$. Indeed, by Equation \eqref{eq:Bbetaj2} and the quantum Serre relations, we have
  \begin{align*}
    [B_{\beta_{N+1}},F_N]_q&=[[F_N,[F_{N-1},B_{\beta_{N-2}}]_q]_q,F_N]_q
      =[F_{N-1},[[F_N,B_{\beta_{N-2}}]_q,F_N]_q]_q = 0
  \end{align*}
  because $F_{N-1}$ and $F_N$ commute. Using the relation $[F_N,B_{\beta_{N-1}}]_q=B_{\beta_{N+1}}$ which also holds by \eqref{eq:Bbetaj2}, we hence obtain for any $m\in \N_0$ the relation
  \begin{align*}
    B_{\beta_{N-1}}F_N^{m+1}v_\lambda&=q^{-1}F_N B_{\beta_{N-1}} F_N^{m}v_\lambda - q^{-1}B_{\beta_{n+1}}F_N^{m}v_\lambda\\
    &=-q^{-1}\sum_{k=0}^{m}(q^{-1})^k F_N^k B_{\beta_{N+1}}F_N^{m-k}v_\lambda\\
    &=-q^{-1}\sum_{k=0}^m q^{m-2k} F_N^m B_{\beta_{N+1}}v_\lambda\\
    &=-q^{-1}[m+1]_q F_N^m B_{\beta_{N+1}}v_\lambda.
  \end{align*}
  By the PBW-Theorem for $M(\lambda)$, Theorem \ref{thm:PBWMl}, the above expression is not zero in $M(\lambda)$.
\end{rema}  

\subsection{Proof of Proposition \ref{prop:FNvl}}\label{sec:FNvl}
We begin with a preparatory Lemma. Recall that by convention $[0]^!_q=1$ and that $[m]^!_q=[m]_q [m-1]^!_q$ for all $m\in \N$.
\begin{lem}\label{lem:BkFm}
  Let $\lambda\in P_{2N-1}$. For all $\ell,m\in \N_0$ with $\ell\le m$ the relation
  \begin{align}\label{eq:BkFm}
    B_{\beta_{N-2}}^\ell F_N^m \,v_\lambda = (-q^{-1})^\ell \frac{[m]^!_q}{[m-\ell]^!_q} F_N^{m-\ell} B_{\beta_N}^\ell\,v_\lambda
  \end{align}
  holds in the Verma module $M(\lambda)$.
\end{lem}  
\begin{proof}
  We proceed by induction on $\ell$. For $\ell=0$ there is nothing to show. Assume that Equation \eqref{eq:BkFm} holds for some $\ell<m$. Using Lemma \ref{lem:BN-2F} we calculate
  \begin{align}
    B_{\beta_{N-2}}^{\ell+1} F_N^m \,v_\lambda &= (-q^{-1})^\ell \frac{[m]^!_q}{[m-\ell]^!_q} B_{\beta_{N-2}}F_N^{m-\ell} B_{\beta_N}^\ell\,v_\lambda\nonumber\\
    &= (-q^{-1})^\ell \frac{[m]^!_q}{[m-\ell]^!_q}\bigg(q^{-(m-\ell)}F_N^{m-\ell} B_{\beta_{N-2}} B_{\beta_N}^\ell\, v_\lambda \label{eq:BFstep}\\
    & \qquad \qquad \qquad- q^{-1}[m-\ell]_q F_N^{m-\ell-1} B_{\beta_N}^{\ell+1} \,v_\lambda \bigg)\nonumber
  \end{align}
  Recall from Lemma \ref{lem:root-vectors} that $B_{\beta_{N}}$ is a root vector of weight $-\gamma_{N-1}$ and that $B_{\beta_{N-2}}=[E_{N-1},B_{\beta_{N-1}}]_q K_{N-1}$ is a root vector of weight $\gamma_{N-2}+\gamma_{N-1}$. Hence
  \begin{align*}
    B_{\beta_{N-2}}B_{\beta_N}^\ell \,v_\lambda \in M(\lambda)_{\lambda-(\ell-1)\gamma_{N-1}+\gamma_{N-2}}=\{0\}
  \end{align*}
  as $\lambda-(\ell-1)\gamma_{N-1}+\gamma_{N-2}\nleq \lambda$. This implies that the first term in the bracket in Equation \eqref{eq:BFstep} vanishes. Hence we have
   \begin{align*}
    B_{\beta_{N-2}}^{\ell+1} F_N^m \,v_\lambda &=(-q^{-1})^{\ell+1} \frac{[m]^!_q}{[m-(\ell+1)]^!_q} F_N^{m-(\ell+1)}B_{\beta_N}^{\ell+1} \,v_\lambda
   \end{align*}
  as desired.   
\end{proof}
Recall the homomorphism of Verma module $\varphi_{N-1}: M(\lambda-(n_{N-1}+1)\gamma_{N-1})\rightarrow M(\lambda)$ from Proposition \ref{prop:Verma-hom}. Lemma \ref{lem:BkFm} has the following consequence.
\begin{cor}\label{cor:in-phiN-1}
  Let $\lambda\in P_{2N-1}$ with $n_{N-1}=\lambda(h_{\gamma_{N-1}})\in \N_0$ and $n_{N-2}=\lambda(h_{\gamma_{N-2}})\in \N_0$ and set $n_N=n_{N-1}+n_{N-2}$. The following relation holds
  \begin{align*}
     B_{\beta_{N-2}}^{n_N+1} F_N^{n_N+1}\,v_\lambda\in \mathrm{Im}(\varphi_{N-1}).
  \end{align*}  
\end{cor}  
\begin{proof}
  By definition, $\mathrm{Im}(\varphi_{N-1})$ is the $\cB_\bc$-submodule of $M(\lambda)$ generated by the element $B_{\beta_N}^{n_{N-1}+1}\,v_\lambda$. By Lemma \ref{lem:BkFm} we have
   \begin{align}\label{eq:BFImphi}
     B_{\beta_{N-2}}^{n_N+1} F_N^{n_N+1}\,v_\lambda=(-q^{-1})^{n_N+1}[n_N+1]^!_qB_{\beta_N}^{n_N+1}\,v_\lambda\in \mathrm{Im}(\varphi_{N-1})
   \end{align}
   as $n_N\ge n_{N-1}$.
\end{proof}  
\begin{proof}[Proof of Proposition \ref{prop:FNvl}]
  For all $i\in X$ one obtains $E_i F_N^{n_N+1}v_\lambda=0$, again by the argument in \cite[Lemma 5.6]{b-Jantzen96}. Recall from Lemma \ref{lem:root-vectors} that $F_N$ is a root vector for $\cB_\bc$ of weight $-\gamma_{N-2}-2\gamma_{N-1}$. Similarly, $B_{\beta_j}$ for $j=1, \dots,N-1$ is a root vector for $\cB_\bc$ of weight $\sum_{\ell=j}^{N-1}\gamma_k$. Using the PBW-Theorem for $\cB_\bc$, Theorem \ref{thm:PBW-Bc}, we conclude that
  \begin{align*}
    (\cB_\bc F_N^{n_N+1}v_\lambda)_\lambda &= \mathrm{span}_\qfield\Big\{B_{\beta_{N-1}}^{j_{N-1}}\dots B_{\beta_1}^{j_1} F_{N}^{n_N+1}v_\lambda\,\big|\\
    & \qquad \qquad \qquad \qquad \sum_{i=1}^{N-1}j_i\sum_{\ell=i}^{N-1}\gamma_\ell = (n_N+1)(\gamma_{N-2}+2\gamma_{N-1})\Big\}\\
    &=\qfield B_{\beta_{N-1}}^{n_N+1} B_{\beta_{N-2}}^{n_N+1} F_N^{n_N+1}v_\lambda.
  \end{align*}
By Corollary \ref{cor:in-phiN-1} we know that $ B_{\beta_{N-1}}^{n_N+1} B_{\beta_{N-2}}^{n_N+1} F_N^{n_N+1}v_\lambda\in \mathrm{Im}(\varphi_{N-1})_{\lambda}=\{0\}$. Hence $v_\lambda \notin \cB_\bc F_N^{n_N+1}v_\lambda$.
\end{proof}

\subsection{Classification of finite-dimensional, irreducible $\cB_\bc$-modules}\label{sec:filt-grad}
Fix $\lambda\in P_{2N-1}^+$. Recall that we use the notation $n_i=\lambda(\gamma_{i})$ for $i=1, \dots, N-1$ and $n_N=\lambda(h_{\gamma_{N-2}}+h_{\gamma_{N-1}})$. Define a submodule $\ttN(\lambda)$ of the Verma module $M(\lambda)$ by
\begin{align*}
  \ttN(\lambda)=\sum_{i=1}^{N-2}\cB_\bc F_{i+1}^{n_i+1}v_\lambda +\cB_\bc F_N^{n_N+1}v_\lambda.
\end{align*}
\begin{rema}\label{rem:NtNtt}
  Equation \eqref{eq:BFImphi} implies that $\im(\varphi_{N-1})\subset \cB_\bc F_N^{n_N+1}v_\lambda$. Indeed, by an $\slfrak(2)$-argument and Equation \eqref{eq:BFImphi} we see that $B_{\beta_{N-1}}^{n_{N-2}} B_{\beta_{N-2}}^{n_N+1} F_N^{n_N+1}v_\lambda$ is a nonzero scalar multiple of $B_{\beta_N}^{n_{N-1}+1}v_\lambda$. For this reason, we do not need to include a summand $\cB_\bc B_{\beta_N}^{n_{N-1}+1}v_\lambda$ in the definition of $\ttN(\lambda)$. 
\end{rema}
By Propositions \ref{prop:Verma-hom} and \ref{prop:FNvl} we have $\ttN(\lambda) \neq M(\lambda)$. Consider the factor module
\begin{align*}
  \ttL(\lambda)=M(\lambda)/\ttN(\lambda).
\end{align*}
We claim that $\ttL(\lambda)$ is finite-dimensional.
\begin{prop}\label{prop:finite}
  For all $\lambda\in P^+_{2N-1}$ the relation $0< \dim(\ttL(\lambda))<\infty$ holds.
\end{prop}
\begin{proof}
  The left-hand inequality holds because $v_\lambda\notin \ttN(\lambda)$. To verify the right-hand inequality define a left ideal $\ttJ_\lambda\subset \cB_\bc$ by
  \begin{align*}
     \ttJ_\lambda =J_\lambda + \sum_{i=1}^{N-2}\cB_\bc F_{i+1}^{n_i+1}  +\cB_\bc F_N^{n_N+1}.
  \end{align*}
By construction we have
\begin{align}
  \ttL(\lambda)=\cB_\bc/\ttJ_\lambda.
\end{align}  
The standard filtration $\cF_\ast$ on $\cB_\bc$ from Section \ref{sec:standardFilt} induces a filtration $\cG_\ast$ on $\ttL(\lambda)$ defined by
\begin{align*}
  \cG_n(\ttL(\lambda))=\cF_n(\cB_\bc)/\ttJ_\lambda\cap \cF_n(\cB_\bc).
\end{align*}
In this way $\ttL(\lambda)$ is a filtered left module over the filtered algebra $\cB_\bc$, see e.g.~\cite[Chapter 6]{b-KrauseLenagan2000}. In particular, the associated graded space $\gr(\ttL(\lambda))$ is a graded module over the graded algebra $\gr(\cB_\bc)\cong\cA$. By the isomorphism theorems we have
\begin{align*}
  \gr_{n}(\ttL(\lambda))=\frac{\cF_n(\cB_\bc)/\ttJ_\lambda\cap \cF_n(\cB_\bc)}{\cF_{n-1}(\cB_\bc)/\ttJ_\lambda\cap \cF_{n-1}(\cB_\bc)}\cong \cA_n \big/\gr_n(\ttJ_\lambda).
\end{align*}
Hence it suffices to show that $\cA/\gr(\ttJ_\lambda)$ is a finite-dimensional vector space.
As $\ttJ_\lambda$ contains the elements
\begin{align*}
  K_i-q^{\lambda(\eta^{-1}(h_{\alpha_i}))}, E_i, B_1, F_j^{n_{j{-}1}+1}, F_N^{n_N+1} \qquad \mbox{for $i\in X$, $j\in X\setminus\{N\}$},
\end{align*}  
the left ideal $\gr(\ttJ_\lambda)\subset\gr(\cB_\bc)\cong \cA$ contains the elements
\begin{align}\label{eq:J'}
   K_i-q^{\lambda(\eta^{-1}(h_{\alpha_i}))}, E_i, F_1, F_j^{n_{j{-}1}+1}, F_N^{n_N+1} \qquad \mbox{for $i\in X$, $j\in X\setminus\{N\}$}.
\end{align}
Let $J'_\lambda\subset \cA$ be the left ideal generated by the elements in \eqref{eq:J'}. As $J'_\lambda\subset \gr(\ttJ_\lambda)$, we see that $\gr(\ttL(\lambda))$ is a quotient of $\cA/J'_\lambda$. By the triangular decomposition \eqref{eq:Atriang} of $\cA$ we have
\begin{align*}
  \cA/J'_\lambda \cong U^-\Big/\Big(U^-F_1+\sum_{j=2}^{N-1} U^- F_j^{n_{j{-}1}+1} + U^- F_N^{n_N+1}\Big).
\end{align*}
Hence the quotient $\cA/J'_\lambda$ is finite-dimensional by \cite[Lemma 5.9]{b-Jantzen96}. This implies that $\gr(\ttL(\lambda))$ and hence $\ttL(\lambda)$ is finite-dimensional.
\end{proof}
\begin{rema} 
  In general, the inclusion $J'_\lambda\subset \gr(\ttJ(\lambda))$ is proper. Indeed, the ideal $\ttJ_\lambda$ contains the root vector $B_{\beta_{N-1}}$ and hence $F_{\beta_{N-1}}\in \gr(\ttJ(\lambda))$ by \eqref{eq:Fbetaj}, \eqref{eq:Bbetaj2}. However, if $n_{N-1}\ge 1$ then $F_{\beta_{N-1}}\notin J'_\lambda$.
\end{rema}  
For any $\lambda\in P_{2N-1}$ the Verma module $M(\lambda)$ has a unique simple quotient $L(\lambda)=M(\lambda)/N(\lambda)$, where $N(\lambda)\subset M(\lambda)$ is the unique maximal proper submodule. By construction $\ttN(\lambda)\subset N(\lambda)$ if $\lambda\in P_{2N-1}^+$ and hence Proposition \ref{prop:finite} implies that $\dim(L(\lambda))<\infty$ in this case. We can now adapt \cite[Theorem 5.10]{b-Jantzen96} to our setting.
\begin{thm}\label{thm:Llambda}
  Let $\qfield$ be any field and $q\in \qfield$ not a root of unity.
  For each $\lambda\in P_{2N-1}^+$ the simple $\cB_\bc$-module $L(\lambda)$ has finite dimension. Each finite-dimensional simple $\cB_\bc$-module (of type ${\bf 1}$) is isomorphic to exactly one $L(\lambda)$ with $\lambda\in P_{2N-1}^+$.
\end{thm}
\begin{proof}
  The first statement has been explained above. The second statement follows from Lemma \ref{lem:hwv}, Corollary \ref{cor:gammaj-dom} and the universal property of the Verma module $M(\lambda)$.
\end{proof}  
\section{Non-restricted specialisation}\label{sec:specialization}
In this section, unless stated otherwise, we assume that $\qfield=\field(q)$ is the field of rational functions in a variable $q$ where $\field$ is a field of characteristic zero. Let $\bA=\field[q]_{(q-1)}$ denote the localisation of the polynomial ring $\field[q]$ with respect to the prime ideal generated by $q-1$. For any $i\in I$ define
\begin{align*}
    (K_i;0)_q = \frac{K_i-1}{q-1}.
\end{align*}
The $\bA$-form $\Uq_\bA$ of $\Uq$ is the $\bA$-subalgebra of $\Uq$ generated by the elements $E_i, F_i, K_i^{\pm 1}$ and $(K_i;0)_q$ for all $i\in I$. As usual, we consider $\field$ as an $\bA$-module via the map sending $q$ to $1\in \field$. The $\field$-algebra $\Uq_1=\field\ot_\bA \Uq_\bA$ is called the specialisation of $\Uq=\uqg$ at $q=1$. For any $x\in \Uq_\bA$ we denote its image in $\Uq_1$ by $\overline{x}$. It is well known \cite[Proposition 1.5]{a-DCoKa90}, \cite[Theorem 3.4.9]{b-HongKang02} that there is an isomorphism of algebras $\Uq_1\rightarrow U(\gfrak)$ given by $\overline{E_i}\mapsto e_{\alpha_1}$, $\overline{F}_i\mapsto f_{\alpha_i}$, $\overline{(K_i;0)_q}\mapsto h_{\alpha_i}$. Here $\gfrak=\glfrak_S(n)_\field$ is defined over $\field$ as in Remark \ref{rem:JantzenJacobson}.
\subsection{$\bA$-forms of triangular decompositions}
For any subspace $W\subseteq \Uq$ we define $W_\bA=W\cap \Uq_\bA$ and $\overline{W}=\field\ot_\bA W_\bA \subset \Uq_1$. By \cite[Proposition 3.3.3]{b-HongKang02} the multiplication map yields an isomorphism
\begin{align}\label{eq:UA-triang}
  \Uq^-_\bA\ot_\bA U^0_\bA \ot_\bA \Uq^+_\bA \cong \Uq_\bA.
\end{align}
We aim to establish a similar $\bA$-form of the triangular decomposition of $\cB_\bc$ in Corollary \ref{cor:Bc-triang}. To this end, the main ingredient is provided by the following Lemma, which is well known, but we include a proof for the convenience of the reader. Recall the definition of the PBW-monomials $F_\cJ$ for $\cJ=(j_1,\dots,j_{N(N-1)})\in \N_0^{N(N-1)}$ from \eqref{eq:FJ-def}.
\begin{lem}
  The algebra $\Uq^-_\bA$  is a free $\bA$-module with basis $\{F_\cJ\,|\,\cJ\in \N_0^{N(N-1)}\}$.
\end{lem}  
\begin{proof}
  Let $V^-_\bA=\oplus_{\cJ\in\N_0^{N(N-1)}} \bA F_\cJ$ denote the free $\bA$-submodule of $U^-$ with basis $\{F_\cJ\,|\,\cJ\in \N_0^{N(N-1)}\}$. As the Lusztig operators $T_i$ for $i\in I$ leave $\Uq_\bA$ invariant, we have $V^-_\bA\subseteq \Uq^-_\bA$. Conversely, $\Uq^-_\bA=U^-\cap \Uq_\bA$ is the $\bA$-subalgebra of $\Uq_\bA$ generated by all $F_i$ for $i\in I$. As $F_i\in V_\bA^-$ it remains to show that $V^-_\bA$ is closed under multiplication. Levedorskii and Soibelman showed that for all $1\le i<j\le N(N-1)$ we have
  \begin{align*}
     F_{\beta_i} F_{\beta_j} - q^{(\beta_i,\beta_j)}F_{\beta_j} F_{\beta_i} = \sum_{\cJ\in \N_0^{N(N-1)}} \bz_\cJ F_\cJ
  \end{align*}
  where $\bz_\cJ\in \Q[q,q^{-1}]$ and $\bz_\cJ=0$ unless $j_r=0$ for $r\le i$ and $r\ge j$; see \cite[Theorem 9.3]{a-DCPro93} for a proof and \cite[I.6.10]{b-BG02} for a discussion of the literature. An inductive argument using Levendorskii and Soibelman's result shows that $F_\cJ F_{\cJ'}\in V^-_\bA$ for all $\cJ,\cJ'\in \N_0^{N(N-1)}$.
\end{proof}
Recall the definition of the monomials $B_\cJ$ for $\cJ\in \N_0^{N(N-1)}$ from \eqref{eq:BJ-def}. Define
\begin{align}
  \cB_{\bc,{\rt}}=\bigoplus_{\cJ\in \N_0^{N(N-1)}} \qfield B_\cJ.
\end{align}
By Proposition \ref{prop:Bc-basis} the multiplication map $\cB_{\bc,\rt}\ot U^0_X \ot \cM_X^+\rightarrow \cB_\bc$ is a $\qfield$-linear isomorphism. The following Lemma is a version of \cite[Theorem 10.7.(4)]{a-Kolb14} where we replaced monomials in the generators $B_i$ by monomials in the root vectors $B_{\beta_i}$.
\begin{thm}\label{thm:Bc-triang-A}
  Let $c_1\in \bA\setminus\{0\}$. The multiplication map
  \begin{align}\label{eq:A-triang-Bc}
     (\cB_{\bc,\rt})_\bA \ot_\bA (U^0_X)_\bA \ot_\bA (\cM_X^+)_\bA \rightarrow (\cB_\bc)_\bA
  \end{align}
  is an isomorphism of $\bA$-modules. Moreover, $\displaystyle (\cB_{\bc,\rt})_\bA=\bigoplus_{\cJ\in \N_0^{N(N-1)}} \bA B_{\cJ}$.
\end{thm}
\begin{proof}
  By Proposition \ref{prop:Bc-basis} and \cite[Lemma 10.6]{a-Kolb14}, to prove the first statement, it suffices to show that the map \eqref{eq:A-triang-Bc} is surjective. Let $b\in \cB_\bc$. By Proposition \ref{prop:Bc-basis} we have $b=\sum_{\cJ\in \N_0^{N(N-1)}} B_\cJ a_\cJ$ for some $a_\cJ\in U_X^0\cM_X^+$. For any $\mu=\sum_{i\in I}n_i\alpha_i\in Q^+$ define $|\mu|=\sum_{i\in I}n_i$ and for $\cJ=(j_1,j_2,\dots,j_{N(N-1)})\in \N_0^{N(N-1)}$ define $|\cJ|=|\sum_{i=1}^{N(N-1)}j_i\beta_i|$. Let $m\in \N_0$ be maximal such that $a_\cJ\neq 0$ for some $\cJ\in \N_0^{N(N-1)}$ with $|\cJ|=m$. By the triangular decomposition \eqref{eq:UA-triang} we can write
  \begin{align*}
     b=\sum_{\cJ\in \N_0^{N(N-1)}, |\cJ|\le m} F_\cJ b_\cJ
  \end{align*}
  for some $b_\cJ\in (\Uq^0 \Uq^+)_\bA$. As $B_1-F_1\in (\Uq^0 \Uq^+)_\bA$ one obtains $a_\cJ=b_\cJ$ if $|\cJ|=m$. Hence $a_\cJ\in (\Uq^0_X\cM_X^+)_\bA$ for all $\cJ\in\N_0^{N(N-1)}$ with $|\cJ|=m$.  By induction on $m$ one obtains $a_\cJ\in (\Uq^0_X\cM_X^+)_\bA$ for all $\cJ\in\N_0^{N(N-1)}$.
  The second statement of the proposition also follows from the above argument.
\end{proof}
Theorem \ref{thm:Bc-triang-A} can be reformulated in terms of the triangular decomposition of $\cB_\bc$ from Corollary \ref{cor:Bc-triang}.
\begin{cor}
  Assume that $c_1\in \bA$. The multiplication map
  \begin{align} \label{eq:Bc-triang-A}
      (\cB_\bc^-)_\bA \ot_\bA (\Uq^0_X)_\bA \ot_\bA (\cB_\bc^+)_\bA \rightarrow (\cB_\bc)_\bA
  \end{align}
  is an isomorphism of $\bA$-modules.
\end{cor}  
\subsection{Specialisation of $\cB_\bc$}\label{sec:Bc-special}
Recall from Equation \eqref{eq:TwXE1} that 
\begin{align*}
  T_{w_X}&(E_1)=T_2\mydots T_{N-2} T_{N-1} T_{N} T_{N-2} \dots T_2(E_1)\\
            &=\big[[E_2,\dots[E_{N-2},E_N]_{q^{-1}}\mydots]_{q^{-1}},[E_{N-1},[E_{N-2}, \mydots[E_2,E_1]_{q^{-1}}\mydots]_{q^{-1}}]_{q^{-1}} \big]_{q^{-1}}.
\end{align*}  
On the other hand Equation \eqref{eq:ms2} gives us
\begin{align*}
  \theta(f_1)=(-1)^{N+1}\big[[e_2,[e_3,\mydots[e_{N-2},e_N]\mydots]],[e_{N-1},[e_{N-2},\mydots[e_2,e_1]\mydots]]\big].
\end{align*}
Hence, the element $B_1=F_1-c_1 T_{w_X}(E_1)K_1^{-1}$ for $c_1\in \cA$ satisfies $\overline{B_1}=f_1+\theta(f_1)$ if and only if $c_1(1)=(-1)^N$. We call an element $c\in \qfield=\field(q)$ specialisable if $c\in \bA$ and $c(1)=(-1)^N$.
\begin{rema}
  In \cite[Section 10.1]{a-Kolb14} we called $c_1\in \bA$ specialisable if $c_1(1)=1$. However, in that paper the definition of the generators $B_i$ involved an additional coefficient $s(\alpha_{\tau(i)})$ which can be chosen to be in $\{\pm 1\}$, see \cite[Remark 3.1]{a-BalaKolb15}, and which we have subsumed in the coefficient $c_1$ in the present paper. 
\end{rema}  
In the following we let $\kfrak=\sofrak(2N)^\theta\cong \sofrak(2N-1)$ be the subalgebra of $\gfrak=\sofrak(2N)$ considered in Section \ref{sec:symPair}. Following Remark \ref{rem:JantzenJacobson} we consider $\gfrak$ and $\kfrak$ defined over the field $\field$.
\begin{thm}[\upshape{\cite[Theorem 4.8]{a-Letzter99a}, see also \cite[Theorem 10.8]{a-Kolb14}}]
If $c_1\in \bA$ is specialisable then $\overline{\cB_\bc}=U(\kfrak)$.   
\end{thm}  
We will now discuss the specialisation of the root vectors $B_{\beta_j}$ for $1\le j\le 2N-2$. Recall the definition \eqref{eq:bi-def} of the elements $b_\ell, d_\ell\in \sofrak(2N-1)$ for $\ell=1,\dots,N-1$.
\begin{prop}\label{prop:root-spec}
  Let $c_1\in \bA$ be specialisable and $1\le \ell\le N-1$. Then we have $B_{\beta_\ell}, B_{\beta_{N-1+\ell}}\in (\cB_\bc)_\bA$ and
  \begin{align}
    \overline{B_{\beta_\ell}}&=\eta(b_\ell),\nonumber\\
    \overline{B_{\beta_{N-1+\ell}}}&=(-1)^{\ell}2\eta(d_{N-\ell}). \label{eq:oBbeta4}
  \end{align}  
\end{prop}  
\begin{proof}
  For specialisable $c_1$ we have seen in Remark \ref{rem:q=1} that $\overline{B_1}=\eta(b_1)$ and more generally $\overline{B_{\beta_\ell}}=\eta(b_\ell)$ for $\ell=1,\dots,N-1$. Moreover, we have seen that $\overline{B_{\beta_N}}=-\eta(f_{\gamma_{N-1}})=-\eta(2d_{N-1})$.
  By Equation \eqref{eq:Bbetaj2} and the third equation of \eqref{eq:2Nin2N1} we have
  \begin{align*}
    \overline{B_{\beta_{N+1}}}=-\eta([f_{\gamma_{N-2}},f_{\gamma_{N-1}}]) =\eta(2d_{N-2}).
\end{align*}
Finally, for $2\le k \le N-2$ the element $T_N\dots T_{k+1}(F_k)$ specialises to
\begin{align*}
  [[[\dots [f_{\alpha_k},f_{\alpha_{k+1}}],&\dots],f_{\alpha_{N-1}}],f_{\alpha_N}]=(-1)^{N-k+1}[E_{N,k}-E_{2N-k+1,N+1},f_{\alpha_N}]\\
  &=(-1)^{N-k+1}(E_{N+2,k}- E_{2N-k+1,N-1})
\end{align*}
in $\sofrak(2N)$.
Hence, for $N+2\le j \le 2(N-1)$ we have
\begin{align*}
  \overline{B_{\beta_j}}&=(-1)^{N-k+1}\big([E_{N+2,k}- E_{2N-k+1,N-1}, \eta(b_{N-2})] \big)\\
  &=(-1)^{N-k+1}\eta\big([E_{N+2,k-1}- E_{2N-k+1,N-2}, E_{N-2,N} - E_{N,N+2}] \big) \\ 
  &=(-1)^{N-k+1}\eta(2d_{k-1})
\end{align*}
where $k=2N-j$. This translates into \eqref{eq:oBbeta4} for $j=(N-1)+\ell$ and $k=N+1-\ell$.
\end{proof}

\subsection{Specialisation of highest weight $\cB_\bc$-modules}\label{sec:mod-special}
Let $V$ be a highest weight $\cB_\bc$-module of highest weight $\lambda\in P_{2N-1}$ with highest weight vector $v_\lambda\in V$. We define
\begin{align*}
  V_\bA= (\cB_\bc)_\bA v_\lambda.
\end{align*}
The triangular decomposition \eqref{eq:Bc-triang-A} of $(\cB_\bc)_\bA$ implies that $V_\bA=(\cB_\bc^-)_\bA v_\lambda$. The negative part $(\cB_\bc^-)_\bA$ of $(\cB_\bc)_\bA$ is a direct sum of $P_{2N-1}$-weight spaces under the adjoint action of $U^0_X$. Moreover, the weight spaces of $(\cB_\bc^-)_\bA$ are finite-dimensional, free $\bA$-modules. Hence
\begin{align*}
  V_\bA= \bigoplus _{\mu\le \lambda} V_{\bA,\mu}
\end{align*}
is a direct sum of $P_{2N-1}$-weight spaces. The weight spaces $V_{\bA,\mu}$ are finitely generated, torsion-free and hence free $\bA$-modules. By \cite[Lemma 10.6]{a-Kolb14} we have
\begin{align}\label{eq:rk=dim}
  \rank_\bA(V_{\bA,\mu}) = \dim_\qfield(V_\mu).
\end{align}
The tensor product $V^1:=\field\ot_\bA V_\bA$ is called the specialisation of $V$ at $q=1$. For any $v\in V_\bA$ we write $\ov=1\ot v\in V^1$.  The specialisation $V^1$ is a module over the specialisation $\field \ot_\bA (\cB_\bc)_\bA=\overline{\cB_\bc}= U(\kfrak)$. By construction we have $\overline{bv}=\ob\,\ov$ for all $b\in (\cB_\bc)_\bA$ and $v\in V_\bA$. For any weight vector $v_\mu\in V_{\bA,\mu}$ and any $h\in Q^\vee_{2N-1}$ we have
\begin{align*}
  \eta(h)\overline{v_{\mu}}= \overline{ (K_{\eta(h)};0)_q}\,\overline{v_\mu}=\overline{ \frac{q^{\mu(h)}-1}{q-1}v_\mu}=\mu(h)\overline{v_\mu}.
\end{align*}
and hence $V^1$ is a $\kfrak$-weight module. Proposition \ref{prop:root-spec} implies that $V^1$ is a highest weight module. Moreover, by \eqref{eq:rk=dim} we have $\dim_\qfield(V_\mu)=\dim_\field(V^1_\mu)$ for all $\mu\in P_{2N-1}$. We summarise the discussion in the following proposition.
\begin{prop}\label{prop:V1}
  Assume that $\qfield=\field(q)$ where $\field$ is a field of characteristic $0$. Let $V$ be a highest weight $\cB_\bc$-module of highest weight $\lambda\in P_{2N-1}$. Then the specialisation $V^1$ of $V$ at $q=1$ is a highest weight $\kfrak$-module of highest weight $\lambda$ and
  \begin{align*}
     \dim_\qfield(V_{\mu})=\dim_\field(V^1_\mu) \qquad \mbox{for all $\mu\in P_{2N-1}$.}
  \end{align*}  
\end{prop}
We can apply Proposition \ref{prop:V1} in particular to the simple $\cB_\bc$-module $L(\lambda)$ from Theorem \ref{thm:Llambda}. Moreover, as in \cite[Theorem 5.15]{b-Jantzen96} we can work over any field $\qfield$ of characteristic $0$ with $q\in \qfield$ transcendental over $\Q$. Recall the definition of the submodules $\widetilde{N}(\lambda)$ and $ \ttN(\lambda)$ over the Verma module $M(\lambda)$ for $\lambda\in P^+_{2N-1}$, and the corresponding quotients $\widetilde{L}(\lambda)$ and $\ttL(\lambda)$, respectively.
\begin{cor}\label{cor:Weyl}
  Assume that $\qfield$ is a field of characteristic $0$ and $q\in \qfield$ is transcendental over $\Q$. Let $\lambda\in P^+_{2N-1}$. Then $\widetilde{L}(\lambda)=\ttL(\lambda)=L(\lambda)$ and the dimension of the weight spaces $L(\lambda)_\mu$ for $\mu\in P_{2N-1}$ are given by Weyl's character formula.
\end{cor}
\begin{proof}
  Assume first that $\qfield=\Q(q)$. Each of $\ttL(\lambda)$ and $L(\lambda)$ are quotients of $\widetilde{L}(\lambda)=M(\lambda)/\widetilde{N}(\lambda)$. Hence the highest weight vector $v_\lambda$ in each of $\widetilde{L}(\lambda)$, $\ttL(\lambda)$ and $L(\lambda)$ satisfies $B_{\beta_N}^{n_{N-1}+1}v_\lambda=0=F_{i+1}^{n_i+1}v_\lambda$ for $i=1,\dots,N-2$   
where $n_i=\lambda(h_{\gamma_i})$ and $n_{N-1}=\lambda(h_{\gamma_{N-1}})$. This means that $f_{\gamma_i}^{n_i+1}\overline{v_\lambda}=0$ for $i=1,\dots,N-1$ in each of the three cases. Hence, each of $\widetilde{L}(\lambda)^1$, $\ttL(\lambda)^1$ and $L(\lambda)^1$ are isomorphic to the unique simple $\kfrak$-module of highest weight $\lambda$. By Proposition \ref{prop:V1} the surjective $\cB_\bc$-module homomorphisms $\widetilde{L}(\lambda)\twoheadrightarrow \ttL(\lambda)\twoheadrightarrow L(\lambda)$ are isomorphisms. The result extends from $\Q(q)$ to $\qfield$ as in \cite[Theorem 5.15]{b-Jantzen96}.
\end{proof}  
We can also use Proposition \ref{prop:V1} to show that every finite dimensional highest weight $\cB_\bc$-module is simple.
\begin{cor} \label{cor:fd-hw}
  Assume that $\qfield$ is a field of characteristic $0$ and $q\in \qfield$ is transcendental over $\Q$.
  Let $V$ be a finite dimensional highest weight $\cB_\bc$-module of highest weight $\lambda\in P^+_{2N-1}$. Then $V\cong L(\lambda)$.
\end{cor}
\begin{proof}
  Again we first consider the case $\qfield=\field(q)$. In this case, by Proposition \ref{prop:V1}, the specialisation $V^1$ is a finite dimensional highest weight $\kfrak$-module of highest weight $\lambda$. Hence $V^1\cong L(\lambda)^1$. Hence Proposition \ref{prop:V1} implies that $\dim_\qfield(V)=\dim_{\qfield}(L(\lambda))$ and as $L(\lambda)$ is a quotient of $V$ we get $V\cong L(\lambda)$. Again we can extend from $\Q(q)$ to a general $\qfield$ of characteristic zero.
\end{proof}
\begin{rema}\label{rem:dual}
  Recall the definition of the dual $V^\ast$ of a finite dimensional $\cB_\bc$-module $V$ from \eqref{eq:dual-def}. As $\kfrak\cong \sofrak(2N-1)$ we have $\dim(L(\lambda)_\mu^1)=\dim(L(\lambda)_{-\mu})$ for all $\mu\in P_{2N-1}$. Hence we have $L(\lambda)\cong L(\lambda)^\ast$ for all $\lambda\in P^+_{2N-1}$ as they are simple $\cB_\bc$-modules with weight spaces of the same dimension. 
\end{rema}  
With Corollary \ref{cor:fd-hw} and Remark \ref{rem:dual} we can show that every finite dimensional $\cB_\bc$-module is semisimple. This can be proved by a word by word translation of the proof of \cite[Theorem 5.17]{b-Jantzen96}.
\begin{thm}\label{thm:semisimple}
   Assume that $\qfield$ is a field of characteristic $0$ and $q\in \qfield$ is transcendental over $\Q$. Then every finite dimensional $\cB_\bc$-module is semisimple.
\end{thm}  
\begin{rema}
  We compare Theorem \ref{thm:Llambda} and the results of the present section with Watanabe's more general approach in \cite{a-Watanabe21}. By Proposition \ref{prop:BcBd-iso} we may assume that $c_1\in\bA$ is specialisable. Watanabe works over the algebraic closure $\qfield=\overline{\C(q)}$ in order to guarantee that Letzter's Cartan subalgebra acts diagonalisably, see \cite[Section 3.1]{a-Watanabe21}. In the specific case of the QSP-coideal subalgebra of type $DII$, the Cartan subalgebra is generated by the elements $K_i^{\pm 1}$ for $i\in X$ which act diagonalisably on finite-dimensional $\cB_\bc$-modules over any field. Hence, in type $DII$, Watanabe's approach also works over $\qfield=\field(q)$ for any field $\field$ of characteristic zero.
  Set
  \begin{align*}
    \cX&=\{B_{\beta_\ell}, E_{\beta_j}\,|\,j=2N-1,\dots,N(N-1), \, \ell=1,\dots,N-1\},\\
    \cY&=\{B_{\beta_\ell},\,|\, \ell=N,\dots,N(N-1)\},\\
    \cW&=\{(K_i;0)_q, K_i^{\pm 1}\,|\, i\in X\}
  \end{align*}
  and let $\cB_{\cX,\bA}$,  $\cB_{\cY,\bA}$ and $\cB_{\cW,\bA}$ be the $\cA$-subalgebras of $\cB_{\bc,\bA}$ generated by $\cX$, $\cY$ and $\cW$, respectively. The triangular decomposition \ref{eq:Bc-triang-A} implies that \begin{align*}
    \cB_{\bc,\bA}= \cB_{\cY,\bA} \cB_{\cW,\bA} \cB_{\cX,\bA}.
  \end{align*}
  This establishes \cite[Conjecture 3.3.3.(1)]{a-Watanabe21} in the case $DII$. The first statement of \cite[Conjecture 3.3.3.(2)]{a-Watanabe21} holds by Lemma \ref{lem:root-vectors}, and the second statement holds as $\overline{(K_i;0)_q}=h_{\alpha_i}$ for all $i\in X$. Finally \cite[Conjecture 3.3.3.(3)]{a-Watanabe21} holds as $\cB_{\cW,\bA}$ is commutative, and \cite[Conjecture 3.3.3.(4)]{a-Watanabe21} is a consequence of Proposition \ref{prop:root-spec}. Hence, \cite[Conjecture 3.3.3]{a-Watanabe21} holds for the example of type $DII$ and the general machinery of \cite[Section 3.3]{a-Watanabe21} can be applied.

  In particular, \cite[Corollary 3.3.10]{a-Watanabe21} provides an alternative proof that
  \begin{align*}
    \dim\big(M(\lambda)/\widetilde{N}(\lambda)\big)<\infty
  \end{align*}
and hence an alternative proof of Theorem \ref{thm:Llambda} if $\mathrm{char}(\qfield)=0$ and $q$ is transcendental over $\Q$. This proof avoids Proposition \ref{prop:FNvl} and the filtered-graded argument of Section \ref{sec:filt-grad}. However, involving specialisation, this proof does not work for fields $\qfield$ of positive characteristic or algebraic $q\in \qfield\setminus\{0\}$ which is not a root of unity. The proofs in Section \ref{sec:mod-special} are analogous to those in \cite{a-Watanabe21}.   
\end{rema}  

\providecommand{\bysame}{\leavevmode\hbox to3em{\hrulefill}\thinspace}
\providecommand{\MR}{\relax\ifhmode\unskip\space\fi MR }
\providecommand{\MRhref}[2]{%
  \href{http://www.ams.org/mathscinet-getitem?mr=#1}{#2}
}
\providecommand{\href}[2]{#2}

\end{document}